\documentclass[11pt]{amsart}
\usepackage{amsfonts,amsmath,amsthm,amssymb,amscd,latexsym,enumerate,etoolbox,mathrsfs,comment, graphicx, enumitem}
\usepackage[all]{xy}
\usepackage{thmtools, thm-restate}
\usepackage{tikz}
\usetikzlibrary{arrows.meta, positioning}
\usepackage[T1]{fontenc}

\newcommand{\R}{\mathbb{R}}
 
\newcommand{\N}{\mathbb{N}}

\newcommand{\Z}{{\mathbb Z}}

\newcommand{\G}{\mathcal{G}}

\renewcommand{\phi}{\varphi}

\newcommand{\into}{\hookrightarrow}
\newcommand{\coker}{\textup{coker}}

\newcommand{\orb}{\mathrm{orb}}

\theoremstyle{plain}
   \newtheorem*{theorem*}{Theorem}
    \newtheorem{theorem}{Theorem}[section]
    \newtheorem{lemma}[theorem]{Lemma}
    \newtheorem{corollary}[theorem]{Corollary}
    \newtheorem{proposition}[theorem]{Proposition}
    
\theoremstyle{definition}
    \newtheorem{definition}[theorem]{Definition}
    \newtheorem{example}[theorem]{Example}

    \newtheorem{remark}[theorem]{Remark}

\theoremstyle{remark}

\DeclareMathOperator{\id}{id}

\DeclareMathOperator{\rank}{rank}

\title[orbit-breaking and Kirchberg algebras]{orbit-breaking in Deaconu--Renault groupoids and models for UCT Kirchberg algebras
}

\author[R.J. Deeley, I.F. Putnam, K.R. Strung]
{Robin J. Deeley \and
Ian F. Putnam \and
Karen R. Strung}
\address{Department of Mathematics,
University of Colorado Boulder
Campus Box 395,
Boulder, CO 80309-0395, USA }
\email{robin.deeley@gmail.com}

\address{Department of Mathematics and Statistics,
University of Victoria,
Victoria, B.C., Canada V8W 3R4} 
\email{ifputnam@uvic.ca}

\address{Institute of Mathematics, Czech Academy of Sciences, \v{Z}itn\'a 25, 115 67 Prague, Czech Republic}
\email{strung@math.cas.cz}

\date{\today}
\thanks{RJD was partially supported by NSF Grant DMS 2247424 and Simons Foundation Gift MP-TSM-00002896. KRS is currently funded by GA\v{C}R project GF25-15403K and \mbox{RVO: 67985840}. IFP is supported in part by an NSERC Discovery Grant.}

\begin{document}

\begin{abstract} 
We generalize the concept of orbit-breaking from the case of a homeomorphism to the case of a surjective local homeomorphism. This method creates open subgroupoids of the Deaconu--Renualt groupoid associated to the surjective local homeomorphism. When the map satisfies natural assumptions (including being locally expanding) the resulting groupoid $\mathrm{C}^*$-algebras are unital UCT Kirchberg algebras. A specific class of Deaconu--Renault systems are constructed, each modelling the Cuntz algebra $\mathcal{O}_2$. The $\mathrm{C}^*$-algebras obtained from orbit-breaking for such systems come with explicit embeddings into $\mathcal{O}_2$. We also introduce a method for orbit-breaking after removing a fixed point. Using this method, we can realize all stable UCT Kirchberg algebras with each coming with an embeddings into $\mathcal O_2\otimes\mathcal K$. Moreover, we show that, given a homomorphism between countable graded abelian groups, there is a $^*$-homomorphism between orbit-breaking algebras obtained from a groupoid inclusion. In particular, we obtain explicit $^*$-homomorphisms between stable UCT Kirchberg algebras from any group homomorphism between their $K$-groups.
\end{abstract}

\maketitle

\section*{Introduction} \label{Sect:Intro}

The Elliott classification program seeks to classify, up to isomorphism, unital, simple, separable, nuclear $\mathrm{C}^*$-algebras via the \emph{Elliott invariant} which consists of $K$-theory and tracial data. The celebrated classification theorem says that this is possible with the additional assumptions that the $\mathrm{C}^*$-algebras are $\mathcal Z$-stable (that is, tensorially absorb the Jiang--Su algebra $\mathcal Z$) and satisfy the Universal Coefficient Theorem (UCT). We will call $\mathrm{C}^*$-algebras which are unital, simple, separable, nuclear, $\mathcal Z$-stable and satisfy the UCT \emph{classifiable}. For an overview of the classification program, we refer the reader to \cite{white2023a, Str:Book}, or the introduction to \cite{CGSTW}.

One of the most pleasing aspects of the theory of $\mathrm{C}^*$-algebras is the way in which other mathematical structures---dynamical systems, graphs, groups, to name a few---can be encoded by $\mathrm{C}^*$-algebras. Since the establishment of the classification theorem, an important goal has been to understand when these $\mathrm{C}^*$-algebras are classifiable. When they are, one may then ask how the Elliott invariant reflects the underlying structure and which invariants can be realized by the given class of constructions. The $\mathrm{C}^*$-algebras arising from topological dynamical systems have, in particular, been of longstanding interest in operator algebras, and can give beautiful examples of classifiable $\mathrm{C}^*$-algebras. 

Kirchberg's dichotomy theorem \cite{Kir:ICM, Phillips:Class} shows that any unital simple exact $\mathrm{C}^*$-algebra which is a tensor product of two infinite-dimensional $\mathrm{C}^*$-algebras is either purely infinite or stably finite. In particular, by virtue of being nuclear and $\mathcal Z$-stable, this is the case for classifiable $\mathrm{C}^*$-algebras. Previous work of the present authors has focused on constructing stably finite classifiable $\mathrm{C}^*$-algebras from free minimal actions of the group of integers by homeomorphisms via crossed products and orbit-breaking algebras \cite{DPSmain}. Although there exist crossed products of this form that are not $\mathcal Z$-stable \cite{GiolKerr},  many such crossed products are classifiable. Moreover, explicit constructions allow for a variety of different invariants to be realized, meaning that many classifiable $\mathrm{C}^*$-algebras admit crossed product models. However, there remain obstructions to realizing \emph{all} classifiable $\mathrm{C}^*$-algebras this way. There are two immediate issues: First, there are $K$-theory obstructions, as the $K_1$-group of such a crossed product is always non-trivial. Second, the crossed product is always stably finite, whereas classifiable $\mathrm{C}^*$-algebras can be both stably finite and purely infinite.

To deal with the $K$-theory obstruction one can use an orbit-breaking construction to produce \'etale groupoids---and in turn, $\mathrm{C}^*$-algebras---that are still very much related to the underlying dynamics, but can realise a wider range of $K$-theory \cite{DPSmain}. Originally introduced by the second author in the study of the $\mathrm{C}^*$-algebras of Cantor minimal systems \cite{MR1194074}, they have become interesting objects of study in their own right. 

However, orbit-breaking $\mathrm{C}^*$-subalgebras of crossed products (by actions of the group of integers or, more generally, by an amenable
group) always admit tracial states, hence are always stably finite. The main purpose of this paper is to address this second issue, that is, to realize purely infinite $\mathrm{C}^*$-algebras.

 To construct purely infinite $\mathrm{C}^*$-algebras, we introduce a version of orbit-breaking that does not begin with a homeomorphism, but rather a surjective local homeomorphism $\varphi : X \to X$ on an infinite compact metric space $X$. From this we construct the Deaconu--Renault groupoid \cite{MR1233967, MR1770333} and its groupoid $\mathrm{C}^*$-algebra.  When $\varphi$ satisfies natural assumptions (including being locally expanding, see Proposition \ref{purelyInfProp}) then the $\mathrm{C}^*$-algebra is simple, separable, unital, nuclear, and, unlike the crossed products by minimal homeomorphisms, purely infinite \cite{A-D:PI}. The resulting orbit-breaking algebras are also purely infinite. We point out that this is in interesting contrast to the orbit-breaking constructions considered by Kettner, where the Cuntz--Pimsner algebra itself is purely infinite, but the orbit-breaking subalgebra is stably finite~\cite[Theorem~6.13]{Kettner2027}. 

 Compared to orbit-breaking for minimal homeomorphisms, in the present setting, the orbit-breaking construction is more subtle and introduces several additional complications. For example, it is not immediately obvious what are the correct notions of the ``forward'' and ``backward'' orbit of a point, or more generally a closed subset. Furthermore, unlike in the case of a minimal homeomorphism, as noted above, the local homeomorphisms we consider necessarily contain eventually periodic points (and in fact, typically fixed points). In this case, extra care is required when orbit-breaking: it is not enough that a closed subset meets every orbit at most once, it must also avoid eventually periodic points. With these conditions in place, the resulting \'etale subgroupoid gives rise to a unital UCT Kirchberg algebra.

 Once the set-up for orbit-breaking has been established, to realize classifiable invariants, we proceed in an analogous manner to the case of orbit-breaking for minimal homeomorphisms \cite[Section 6]{DPSmain}. More precisely, using results of the second author, we construct a system $(X, \varphi)$ whose groupoid $\mathrm{C}^*$-algebra is isomorphic to $\mathcal O_2$, so in particular has trivial $K$-theory. Moreover, we are able to arrange that $X$ has connected components of arbitrarily large finite dimension. This allows us to embed any finite-dimensional compact metric space $Y$ in a way that permits orbit-breaking at its image in $X$. In this way we construct unital UCT Kirchberg algebras that have $\mathbb Z$ as a direct summand in $K_0$ (we also lack control of the order of the unit in the $K_0$-group).

We are also able to orbit-break at locally compact spaces to realize non-unital $\mathrm{C}^*$-algebras. This is done by embedding the one-point compactification of a locally compact non-compact set $Y$ into $X$ in such a way that $Y$ contains no eventually periodic points, meets every orbit at most once, and the point at infinity is mapped to a fixed point. After removing the fixed point from the system, we break orbits at $Y$. This orbit-breaking process lead to the construction of all stable UCT Kirchberg algebra.

Although our construction is completely novel, it is not the first construction of groupoid models for UCT Kirchberg algebras. Indeed, there have been many such constructions, beginning with Spielberg's graph based models \cite{Spielberg:Kirch}, followed by Katsura's realization via topological graph $\mathrm{C}^*$-algebras \cite{KatsuraPI}. In a purely groupoid setting,  Orloff Clark, Fletcher, and an Huef show that UCT Kirchberg algebras have ample groupoid models \cite{ClarkFletcheranHuef}. More recently, Wu models Kirchberg algebras via directed graphs of groups \cite{Wu:KirchModels}, while Evington and Sibbel revisit topological graph $\mathrm{C}^*$-algebras  to construct principal groupoid models \cite{EvingtonSibbel1,EvingtonSibbel2}.

Among these constructions, the orbit-breaking technique is especially satisfying as the orbit-breaking subalgebras arise from open subgroupoids of the groupoid models for $\mathcal O_2$ and $\mathcal O_2 \otimes \mathcal K$, and, as such, the resulting UCT Kirchberg algebras come equipped with explicit embeddings into $\mathcal O_2$ in the unital case and $\mathcal O_2 \otimes \mathcal K$ in the stable case. In fact, we can do even better than this: given any group homomorphism between $\mathbb{Z}/2\mathbb{Z}$-graded countable abelian groups, we produce maps between the stable UCT Kirchberg algebras with those $K$-theory groups, which arise from open groupoid inclusions. In other words, we
are able to realize morphisms of the invariants at the groupoid level.

The paper is structured as follows. In Section~\ref{Sect:Prelim}, we recall some properties of \'etale groupoids before specializing to the Deaconu--Renault groupoid arising from a surjective local homeomorphism. We look at the specific example of a one-sided shift of finite type given by a directed graph. Next, we recall the construction of binary factors of shifts of finite type, due to the second author, which will allow us to construct a model of $\mathcal O_2$, having a Cartan subalgebra whose spectrum is not a Cantor set, in which we perform orbit-breaking. In Section~\ref{sec:OBDR}, we introduce orbit-breaking subgroupoids and prove the existence of a six-term exact sequence which will allow us to compute $K$-theory. Section~\ref{sec:lcgpi} proves that under natural conditions on the map (including being locally expanding) the orbit-breaking subalgebras of the Deaconu--Renault groupoids are simple and purely infinite, hence UCT Kirchberg algebras. We construct models for $\mathcal O_2$ in Section~\ref{sec:OBO2}. With a view towards embedding arbitrary finite-dimensional metric spaces, these models are required to have connected components of arbitrarily large dimension. In the final two sections, we apply the orbit-breaking method to the systems from Section~\ref{sec:OBO2}. In particular, we realize all stable UCT Kirchberg algebras via orbit-breaking.

\subsubsection*{Acknowledgments} This project was initiated at the 2024 BIRS workshop \emph{Cartan Subalgebras in Operator Algebras, and Topological Full Groups}. The authors thank the organizers and the staff at BIRS. RJD and KRS wish to thank the University of Victoria for their visit in November 2024. KRS is grateful to the University of Colorado Boulder for visits in 2024, 2025, and 2026. Further work on the project was done at the NSF/CBMS Regional Conference at Texas Christian University in June 2025, and at the Canadian Operator Symposium (COSy) in June 2026.

\subsubsection*{AI statement} GPT-5.6 Sol was used to assist with preparing the bibliography. No LLMs were used for the mathematical arguments, which were solely produced by the authors.

\section{Preliminaries} \label{Sect:Prelim}

In this section, we set notation and gather the necessary preliminaries on \'etale groupoids, in particular Deaconu--Renault groupoids, as well as dynamical systems associated to local homeomorphisms. We also recall some of the results of \cite{MR4700629} we will need for the construction of our groupoid model of $\mathcal O_2$. 

\subsection{Groupoids} For a good introduction to \'etale groupoids and their $\mathrm{C}^*$-algebras, we refer the reader to Nekrashevych \cite{MR4475129}, Renault's book \cite{MR584266}, and the chapter by Sims in \cite{MR4321941}.

Suppose $\mathcal{G}$ is a second countable locally compact \'etale groupoid. The unit space of $\mathcal{G}$ is denoted by $\mathcal{G}^{(0)}$. 
\begin{definition} \label{def:Gset}
A \emph{$\mathcal{G}$-set} (sometimes called a \emph{bisection}) is a subset $S \subset \mathcal G$  where the restriction of the range and source maps to $S$ are homeomorphisms onto their images.  Given a $\mathcal G$-set $S$, the map
\[
\alpha_S: r(S) \rightarrow s(S) \hbox{ defined via }x \mapsto s(\gamma)
\]
is a homeomorphism where $\gamma$ is the (unique) element in $S$ with $r(\gamma)=x$.
\end{definition}

Note that if $S$ is a $\mathcal G$-set, then so is $S^{-1}=\{ \gamma^{-1} \mid \gamma \in S \}$. Moreover, the map $\alpha_S^{-1}$ is equal to $ \alpha_{S^{-1}}$.

\begin{definition} \label{def:Gorbit}
Given $x \in \mathcal{G}^{(0)}$, the groupoid orbit of $x$ is the set $r(s^{-1}(x))$.
\end{definition}

The next lemma is a well-known result, see for example \cite[page 105]{MR2648649}.

\begin{lemma}
Suppose that $\mathcal{G}_1$ and $\mathcal{G}_2$ are amenable \'etale groupoids. Then $\mathcal{G}_1 \times \mathcal{G}_2$ is an amenable \'etale groupoid and $C^*(\mathcal{G}_1\times \mathcal{G}_2) \cong C^*(\mathcal{G}_1)\otimes C^*(\mathcal{G}_2)$.
\end{lemma}

Recall that when a groupoid is amenable, the full and reduced groupoid $\mathrm{C}^*$-algebras coincide and are nuclear. Thus the choice of completion of the groupoid $\mathrm{C}^*$-algebras, as well as the choice of norm on the tensor product, are not indicated.

 Given a locally compact \'etale groupoid $\mathcal G$ and an open set $U \subset \mathcal G^{(0)}$, the \emph{reduction} of $\mathcal G$ to $U$ is given by
\[ \mathcal G|_U  := \{ \gamma \in \mathcal G \mid s(\gamma) \in U \text{ and } r(\gamma) \in U \}.\]
Notice that $\mathcal G|_U$ is a locally compact \'etale groupoid with unit space $\mathcal G|_U^{(0)} \cong U$ and multiplication and inverse given by restriction. This makes $\mathcal G|_U$ into an open \'etale subgroupoid of $\mathcal G$.

\subsection{Deaconu--Renault groupoids} 

Let $X$ be a compact metric space and $\varphi: X \rightarrow X$ be a surjective local homeomorphism. With a small abuse of notation, we define, for $x \in X$ and $l \in \mathbb{N}$,
\[
\varphi^{-l}(x) = \{ y \in X \mid \varphi^{l}(y) = x\}.
\]
 The backward orbit of a point $x \in X$ is given by 
\[ \orb^-(x) := \{ y \in X \mid \varphi^k(y) = x \text{ for some } k \geq 0\},
\]
while the forward orbit is
\[ \orb^+(x) := \{ \varphi^k(x) \mid k \geq 1 \}.\]
Note that our convention includes $x$ as a point in its backward orbit, rather than the its forward orbit. Unlike the case of a homeomorphism, the correct notion of the orbit of a point is not simply the union of forward and backward orbits, as we require that if $y$ is in the orbit of $x$, then the orbits of $x$ and $y$ should coincide. However, if $y$ is in the forward orbit of $x$, $\orb^+(x) \cup \orb^-(x)$ will not necessarily contain $\orb^-(y)$. Thus the correct notion of the (generalized) orbit of $x$ is given by the set
\[ \orb(x) :=\bigcup_{l, k \in \N\cup\{0\}} \varphi^{-l}(\varphi^k(x)).
\]
Equivalently, the orbit of $x$ is the increasing union 
\[
\orb(x) = \bigcup_{k \geq 0} \orb^-(\varphi^k(x)).
\]
Note that if $\varphi$ is a homeomorphism $\varphi^{-1}(x)$ contains only one element  for every $x\in X$ and this is just the usual notion of orbit. However, for the maps considered below, we will often have $x \in X$, such that the preimage set $\varphi^{-1}(x)$ contains at least two points. In fact, for the specific maps we construct, this property will hold for all $x\in X$. Indeed, the systems constructed in the proof of Theorem \ref{thm:dimO2} have this property. Example \ref{ex:OnePreSFT} gives a system that does not have this property.

We say that $x\in X$ is \emph{eventually periodic} if there exists $n \neq m > 0$ such that $\varphi^n(x)=\varphi^m(x)$. This is equivalent to the forward orbit of $x$ being finite. The set of eventually periodic points is denoted by $X_{{\rm ep}}$.

\begin{example} \label{ex:twofold} 
Let $X=S^1$ be the unit circle in the complex plane and $\varphi: S^1 \rightarrow S^1, z \to z^2$ be the two-fold covering map. Then $x$ is an eventually periodic point if and only if $x$ is a root of unity. Equivalently, identifying $S^1$ with $\mathbb R/\mathbb Z$ via $t\mapsto e^{2\pi i t}$, the map becomes
\[
\varphi:\mathbb R/\mathbb Z\to\mathbb R/\mathbb Z,\qquad z\mapsto 2z \pmod{\mathbb Z},
\]
and the eventually periodic points are precisely the elements of $\mathbb Q/\mathbb Z$. See Figure~\ref{fig:GenOrbit} for the generalized orbit of a point $z \in S^1$ which is not eventually periodic. \end{example}

\begin{figure}[ht]
    \centering
    \includegraphics[width=0.8\textwidth]{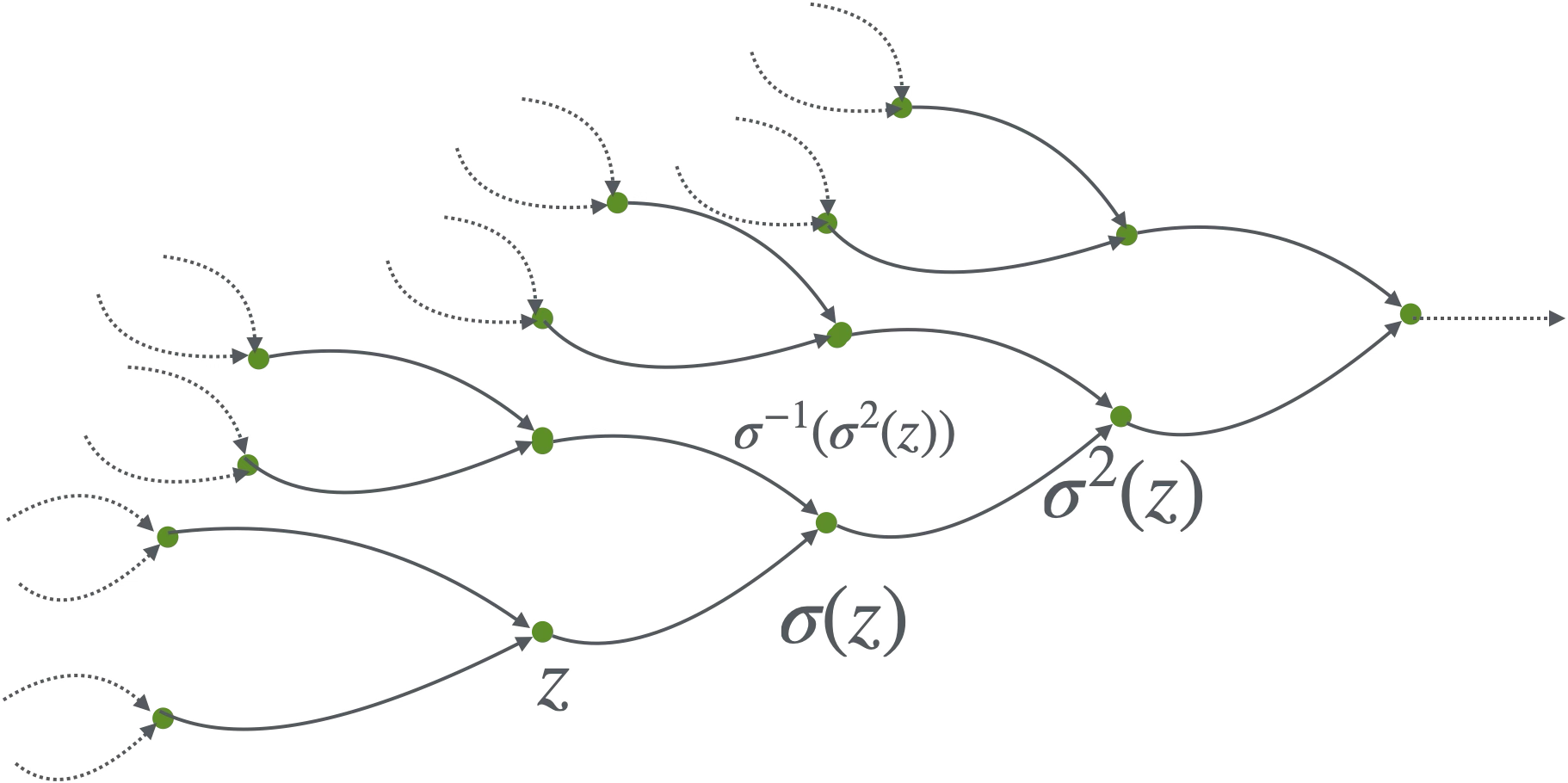}
    \caption{Generalized orbit of $z \in S^1$ that is not a root of unity under the covering map $\varphi : S^1 \to S^1,  z \mapsto z^2$.}
    \label{fig:GenOrbit}
\end{figure}

The \emph{Deaconu--Renault groupoid} \cite{MR1233967, MR1770333} associated to $(X, \varphi)$ is given by
\[
\mathcal{G} = \{ (x, m-n, z) \in X\times \Z \times X \mid m,n \in \N \hbox{ and }\varphi^m(x)=\varphi^n(z) \}.
\]
Two elements $(x, p, y), (w, q, z) \in \mathcal{G}$ can be composed if $y=w$, in which case
\[
(x,p,y)\cdot (y, q, z)=(x, p+q, z).
\]
The unit space is $\mathcal G^{(0)} = \{ (x, 0, x) \mid x\in X\}$, which is identified with $X$. The range and source maps are given respectively by 
\[
r(x,p,z)=z \hbox{ and }s(x,p,z)=x.
\]
The groupoid $\mathcal{G}$ is \'etale with $\mathcal{G}$-sets obtained as follows. Given an element $(x_0, m-n, z_0) \in \mathcal{G}$, let $U$, $V$ be open subsets of $X$ such that 
\begin{enumerate}
\item $x_0\in U$ and $z_0\in V$,
\item $\varphi^m(U)=\varphi^n(V)$,
\item $\varphi^m|_U$ and $\varphi^n|_V$ are homeomorphisms onto their common image.
\end{enumerate}
Then $\{ (x, m-n, z) \mid x\in U \hbox{ and }z\in V \hbox{ with }\varphi^m(x)=\varphi^n(z) \}$ is a $\mathcal{G}$-set. 

It is worth noting that the orbit of a point $x \in \mathcal G^{(0)}$ in the groupoid sense (Definition~\ref{def:Gorbit}) is the same as the set $\orb(x)$ defined at the start of this section. That is, 
\[
r ( s^{-1}(x, 0, x)) =  \orb(x)=\bigcup_{l, k \in \N\cup\{0\}} \varphi^{-l}(\varphi^k(x)).
\]

\begin{proposition}	\label{prop:GUopen}
Let $X$ be a compact metric space and $\varphi: X \to X$ a continuous, open surjection and $w \in X$  a fixed point. Denote by $\mathcal G$ the associated Deaconu--Renault groupoid and let $U := X \setminus \{w\}$. Then $\mathcal G|_U$ is an amenable \'etale groupoid. Moreover, $C^*(\mathcal G|_U)$ is a hereditary $\mathrm C^*$-subalgebra of $C^*(\mathcal G)$.
\end{proposition}

\begin{proof}
    The reduction $\mathcal G|_U$ is an open subgroupoid of the \'etale groupoid $\mathcal G$, so it is also \'etale. It is well known that $\mathcal{G}$ is amenable (see for example \cite[Example 10.1.12]{MR4321941}) and open subgroupoids of amenable groupoids are amenable (see for example, \cite[Proposition 10.1.14]{MR4321941}). Let $a \in C(\mathcal G^{(0)})$ be any function satisfying $a (w) = 0$ and $a(x) >0$ for every $x \neq w$. Then, for any $f \in C_c(\mathcal G)$ we have $a f a(\gamma) = a(r(\gamma))f(\gamma) a(s(\gamma)) = 0$ unless $r(\gamma), s(\gamma) \in U$. Thus $C^*(\mathcal G|_U)$ is a hereditary $\mathrm{C}^*$-subalgebra.
\end{proof}

\subsection{Shift spaces and locally expanding maps}

To construct the dynamical system of Theorem~\ref{thm:dimO2}, some background from \cite{MR4700629} is required. The starting point is shifts of finite type, which are a special case of locally expanding maps. The book \cite{LindMarcus:SDandCod} contains a detailed introduction to shifts of finite type. 

Any shift of finite type is topologically conjugate to an edge shift on a finite directed graph \cite[Theorem 2.3.2]{LindMarcus:SDandCod}.  A directed graph is denoted by $E =(E^0, E^1, i, t)$ where $E^0$ is the set of vertices, $E^1$ is the set of edges, and $i$ and $t$ are the source and range maps respectively.  The adjacency matrix of $E$ is denoted by $A_E$ or just $A$ if we are only discussing a single graph.  

A graph $E$ with adjacency matrix $A_E$ is \emph{primitive} if there is a positive integer $k$ such that $A_E^k(v,w)$ is positive for all $v,w \in E^0$ where $A_E^k(v,w)$ denotes the $(v,w)$-entry of the $k^{\text{th}}$ power of $A_E$.

Let $E$ be a finite directed graph with adjacency matrix $A$. The associated one-sided shift space is given by 
\[ \Sigma^+_A= \{ (e_n)_{i\in \N} \in (E^1)^\mathbb{N} \mid t(e_n)=i(e_{n+1}) \}, \]
together with the left shift
\[ \sigma_+((e_n)_{n\in \N}) = (e_{n+1})_{n \in \mathbb N}. \]
The one-sided shift is a surjective local homeomorphism.

\begin{definition} \label{def:irred}
Suppose $X$ is a compact metric space and $\varphi: X \rightarrow X$ is a continuous, open, surjection. Then $(X, \varphi)$ is
\begin{enumerate}
\item \emph{irreducible} if, for each ordered pair of non-empty open sets $U$ and $V$, there exists $n\in \N$ such that $\varphi^n(U) \cap V$ is non-empty, and 
\item \emph{mixing} if, for each ordered pair of non-empty open sets $U, V \subset X$ there is $N \in \mathbb N$ such that $\varphi^n(U) \cap V \neq \emptyset$ for every $n \geq N$.
\end{enumerate}
\end{definition}
It is worth noting that mixing implies irreducible, but the other direction does not hold. 
\begin{definition} \label{def:ExpMap}
Suppose $X$ is a compact metric space and $\varphi: X \rightarrow X$ is a continuous, open, surjection. Then $\varphi$ is
\begin{enumerate}
    \item  \emph{expanding} if there exists $\gamma>0$ such that if $d(\varphi^n(x), \varphi^n(z))< \gamma$ for all $n\in \N\cup\{0\}$, then $x=z$,
    \item \emph{locally expanding} if there exist constants $\delta_X>0$ and $\lambda_X>1$ such that $\lambda_X d(x, z) \le d(\varphi(x), \varphi(z)),$
for any $x, z \in X$ with $d(x,z)<\delta_X$.
\end{enumerate}
\end{definition}

It follows from work of Reddy \cite{MR651514} that an expanding map $\varphi$ is a local homeomorphism if it is an open map. Reddy \cite{MR651514} also proved that if $\varphi$ is expanding, then there exists a metric $d$ on $X$ (giving the same topology as the original metric) for which $\varphi$ is locally expanding. 

The dynamical systems we consider in the present paper are always open maps. In Section \ref{sec:OBDR} we will only assume the map is a surjective, local homeomorphism. Then, starting in Section \ref{sec:lcgpi} we will assume the map is also locally expanding. Based on Reddy's result one can apply our results to expanding maps (rather than locally expanding maps) by changing the metric on the space as in Reddy's paper \cite{MR651514}. Even later, in Sections \ref{sec:OBO2} to \ref{sec:OBO2LC}, we will only be considering the specific systems constructed in \cite{MR5002161} (which are surjective, open, locally expanding, and have other nice properties).

In particular, the natural solenoid associated to $(X, \varphi)$ where $\varphi$ is open and (locally) expanding is a Smale space \cite[Section 7.26]{Rue:ThermForm}. In the present paper, Smale space theory is only used in the proofs of Theorems \ref{thm:PutSys} and \ref{thm:productDR}, which only involve the specific systems constructed in \cite{MR5002161}. The interested reader can also see Remark \ref{rem:irr-mix} for how Smale space theory is used.

The two-fold self-cover of the circle discussed in Example \ref{ex:twofold} is an example of an open locally expanding map where the system is also irreducible and also mixing. All one-sided shifts of finite type are locally expanding and open. The irreducibility of a shift of finite type obtained from a (finite, directed) graph is easy to determine from the graph. In particular, the shift is irreducible if and only if, for each ordered pair of vertices, $v$ and $w$, there exists a finite path from $v$ to $w$. One also has that a one-sided edge shift is mixing if and only if its adjacency matrix is primitive. 

\begin{lemma} \label{lem:denseBack}
Suppose that $(X, \varphi)$ is irreducible, open, and locally expanding. Then, for every $x \in X$, the backward orbit $\orb^-(x)$ is dense in $X$. 
\end{lemma}
\begin{proof}
Consider first the special case of an irreducible one-sided shift of finite type, which we denote by $(\Sigma^+, \sigma_+)$. As mentioned above without loss of generality, we can assume that there is a finite directed graph, $E$, such that
\[ 
\Sigma^+ = \{ (e_1, e_2, \ldots ) \mid (e_1, e_2, \ldots) \hbox{ is a one sided infinite path in }E \},
\]
and $\sigma_+$ is the shift map
\[
(e_1, e_2, \ldots) \mapsto (e_2, \ldots ).
\]
Now, fix $x_+=(x_1, x_2, \ldots)$ and $w_+=(w_1, w_2, \ldots) \in \Sigma^+$ along with $K \in \N$. We must find $z_+=(z_1, z_2, \ldots ) \in \Sigma^+$ and $n\in \N$ such that 
\[
\sigma_+^n(z_+)=x_+ \hbox{ and }z_i=w_i \hbox{ for }i=1, \ldots, K.
\]
Since $(\Sigma^+, \sigma_+)$ is irreducible, there exists a finite path $a_1, a_2, \ldots, a_L$ such that $t(w_N)=i(a_1)$ and $t(a_L)=i(x_1)$. As such, 
\[ z_+=(w_1, \ldots , w_N, a_1, \ldots, a_L, x_1, x_2, \ldots ) \in \Sigma^+.
\]
One readily checks that this choice of $z$ has the two required properties. This completes the proof of the special case of irreducible one-sided shifts of finite type.

In the general case, by the one-sided version of Bowen's theorem (see \cite[Theorem 6.5]{Adler:SD} or \cite[Sections 29 and 30]{Rue:ThermForm}), there exists an irreducible one-sided shift of finite type $(\Sigma^+, \sigma_+)$ and a factor map $\pi : (\Sigma^+, \sigma_+) \rightarrow (X, \varphi)$. Let $x\in X$ and choose $x_+ \in \Sigma^+$ such that $\pi(x_+)=x$. Then, since $\pi \circ \sigma_+ = \varphi \circ \pi$, 
\[
\pi(\orb^-(x_+)) \subseteq \orb^-(x).
\]
Since $\pi$ is continuous, surjective, and $\orb^-(x_+)$ is dense in $\Sigma^+$, we have that $\orb^-(x)$ is dense in $X$.
\end{proof}

\begin{corollary} \label{cor:denseBackLC}
        Suppose that $(X, \varphi)$ is irreducible, open, surjective, locally expanding, and contains a fixed point, $w$. Then, for every $x \in X \setminus \{w\}$, the backward orbit of $x$ is dense in $X \setminus \{w\}$. 
\end{corollary}

\begin{proof}
    Immediate from Lemma~\ref{lem:denseBack}.
\end{proof}

Many locally expanding maps have the property that each point has at least two preimages (for example, the two-fold cover of the circle). However, this property is not completely general for the systems we are interested in. In Lemma \ref{lem:Preimages} we prove a weaker condition holds. Later this dynamical property will be used to show that the orbit-breaking groupoid is minimal and hence that the associated $\mathrm{C}^*$-algebra is simple. The  proof of the following lemma is straightforward and so omitted.  

\begin{lemma} \label{lem:PreimagesSFT}
    Let $A = (a_{ij})_{i,j}$ be the adjacency matrix of a finite directed graph with $n$ vertices and $(\Sigma^+, \sigma_+)$ the corresponding one-sided shift space. Given $(e_k)_{k\in \mathbb N} \in \Sigma^+$, we have
    \[ |\sigma_+^{-1}((e_k)_{k \in \mathbb N})| = \sum_{j=1}^n a_{ij},\]
    where $i = i(e_1)$.
\end{lemma}

\begin{example} \label{ex:OnePreSFT}
    Let 
    \[A=\left[\begin{array}{cc}
       0  & 1 \\
       1  & 1
    \end{array} \right].\]
    Then Lemma \ref{lem:PreimagesSFT} implies that the associated one-sided shift of finite type has a point whose preimage set is a single point. We note that the entries of $A^2$ are all positive, so the associated shift of finite type is mixing.
\end{example}

\begin{lemma}
 Let $A$ be the adjacency matrix of a finite directed graph with $n$ vertices and $(\Sigma^+, \sigma_+)$ the corresponding one-sided shift space. Suppose that $(\Sigma^+, \sigma_+)$ is mixing and not trivial (that is, $A \neq [1]$). Then there exists $N \in \mathbb{N}$ such that
 \[
 |\sigma_+^{-N}((e_k)_{k \in \mathbb N})| \geq 2,
 \]
  for every  $(e_k)_{k \in \mathbb N} \in \Sigma^+_A$.
\end{lemma}

\begin{proof}
    Suppose first that $A = [a]$ is a $1 \times 1$ matrix.  Since $A$ is non-trivial, $a \geq 2$, and the result follows from the Lemma~\ref{lem:PreimagesSFT}.  Otherwise, let $A = [a_{ij}]_{i,j}$ be an $n \times n$ matrix with $n \geq 2$. Since $(\Sigma_A^+, \sigma_+)$ is mixing, there exists $N \geq 0$ such that $(A^N)_{i,j} \geq 1$ for every $1 \leq i,j \leq n$. The result follows from Lemma~\ref{lem:PreimagesSFT}.
    \end{proof}

    \begin{theorem} \label{lem:Preimages}
        Let $X$ be an infinite compact metric space with no isolated points and $\varphi: X \to X$ an open, locally expanding, continuous surjection. Suppose that $(X, \varphi)$ is mixing. Then, for every $x \in X$ there exists $N > 0$ such that 
        \[
      |  \varphi^{-N}(x) | \geq 2.
    \]
    \end{theorem}

    \begin{proof}
         Theorem 6.5 of \cite{Adler:SD} (also see \cite[Sections 29 and 30]{Rue:ThermForm}) implies there exists a mixing shift of finite type, $(\Sigma^+, \sigma_+)$ and a factor $\pi : (\Sigma^+, \sigma_+) \to (X, \varphi)$ that is one-to-one almost everywhere. In particular, this map has the following property: given a non-empty open set $U \subseteq X$, there exists $z \in U$ such that $|\pi^{-1}(\{z\})| = 1$. 
        
        By the previous lemma, there is $N\in\N$ such that we have $|\sigma_+^{-N}(a)| \geq 2$ for any $a\in \Sigma^+$. 
        
        We will show that for all $x\in X$, $|\varphi^{-N}(x)|\ge 2$. Seeking a contradiction, suppose that $x\in X$ and suppose  $\varphi^{-N}(x)=\{y\}$. Find open sets $U, V \in X$ such that 
        \begin{enumerate}
        \item $\varphi^N|_U : U \to V$ is injective, and
        \item $y \in U$ and $x \in V$. 
        \end{enumerate}
        Since $y$ is the unique preimage (and by shrinking $V$ if necessary) we can assume that $\varphi^{-N}(V) \subseteq U$. Let $\hat{y} \in \varphi^{-N}(V)$. Then $\varphi^N(\hat{y}) \in V$ and so $\varphi^{-N}(\varphi^N(\hat{y})) \subset \varphi^{-N}(V)$, so $\varphi^{-N}(\varphi^N(\hat{y})) = \{ \hat{y} \}$. 
        
        We now make a more specific choice of $\hat{y}$. Namely, take $\hat{y} \in \varphi^{-N}(V)$ with $\pi^{-1}(\hat{y})=\{a\}$ for some $a \in \Sigma^+$. Then $\pi^{-1}(\varphi^N(\hat{y})) = \{ \sigma_+^n(a) \}$ and $|\sigma_+^{-N}(\sigma_+^N(a))| \geq 2$. Let $b \in \sigma_+^{-N}(\sigma_+^N(a))$ with $b \neq a$. Now $\pi(b) \neq \pi(a) = \hat{y}$ since $\pi^{-1}(\hat{y}) = \{a \}$. But $\varphi^N(\pi(b)) = \pi(\sigma_+^N(b)) = \pi (\sigma_+^N(a)) = \varphi^N(\hat{y})$. Since $\pi(b) \neq \pi(a)$, we have that $| \varphi^{-N}(\varphi^N(\hat{y}))| \geq 2$, contradicting the fact that $\varphi^{-N}(\varphi^N(\hat{y})) = \{ \hat{y} \}$. Thus $|\varphi^{-N}(x)| \geq 2$.
    \end{proof}

    \begin{remark} \label{rem:irr-mix}
        For readers familiar with Smale space theory, we note that the previous theorem generalizes to the case when $(X, \varphi)$ is irreducible (and in fact  nonwandering). The proof uses a variant of Smale's decomposition theorem \cite[Section 31]{Rue:ThermForm} to reduce to the mixing case. However, since this generalization is not required for the results of the present paper, we will not go into further details. 
    \end{remark}

\subsection{Binary factors of shift of finite type}

Later we will make use of a class of dynamical systems constructed by the second author in \cite{MR4700629}. The reader can also see \cite{MR5002161} where this construction is also discussed. In particular, although not relevant for the systems considered in the present paper, property (H2) and (H3) below can be weakened by \cite[Section 4.1]{MR5002161}. 

The relevant systems are constructed from two directed graphs, $F$ and $E$, along with graph embeddings $\xi^0 : F \rightarrow E$ and $\xi^1: F \rightarrow E$ satisfying the following:
\begin{enumerate}
\item[(H0)] $(\xi^0)|_{F_0} = (\xi^1)|_{F_0}$,
\item[(H1)] $\xi^0(F_1) \cap \xi^1(F_1)= \emptyset$,
\item[(H2)] for each $e \in F^1$, there exists $f\in E^1$ with $t(f)=t(\xi^0(e))$, $i(f)=i(\xi^0(e))$ and $f\not\in \xi^0(F^1) \cup \xi^1(F^1)$, and
\item[(H3)] the graph $E$ is primitive. 
\end{enumerate} 
Given these inputs, the output is a compact metric space, $X^+_{\xi}$, and a local homeomorphism $\sigma_{\xi} : X^+_{\xi} \rightarrow X^+_{\xi}$ along with a factor map, $\pi_{\xi} : X^{+}_{E} \rightarrow X^{+}_{\xi}$. A specific example is considered in detail in Example \ref{ex:IanBasEx}. The precise properties of the system $(X^+_{\xi}, \sigma_{\xi})$ are summarized in the next theorem. We will call such a system the \emph{binary system associated to $(\xi^0, \xi^1)$}.

Recall that, given a dynamical system $(X, \varphi)$, $X_{\mathrm{ep}}$ denotes the set of eventually periodic points (where a point $x \in X$ is eventually periodic if there exists $n, m \geq 0$, $n\neq m$, such that $\varphi^n(x) = \varphi^m(x)$).

\begin{theorem} \label{thm:PutSys}
Suppose $F, E$ are finite directed graphs with graph embeddings $\xi^0, \xi^1 : F \to E$ satisfying (H0)--(H3). Let $(X^+_{\xi}, \sigma_{\xi})$ be the binary system associated to $(\xi^0, \xi^1)$.  Then the following hold.
\begin{enumerate}
\item The space $X^+_{\xi}$ has no isolated points, $\sigma_\xi$ is mixing, open, and locally expanding, and $(X^+_{\xi})_{\mathrm{ep}}$ is dense but has empty interior. 
\item The connected components of $X^+_{\xi}$ are circles and points and both types occur. 
\item The groupoid $\mathrm{C}^*$-algebra associated to $(X^+_{\xi}, \sigma_{\xi})$ is a  unital UCT Kirchberg algebra with $K$-theory given by
\begin{align*}
\frac{\Z^{E^0}}{(I-A^T_E)\Z^{E^0}} \oplus {\rm ker}(I-A^T_F) & \hbox{ in degree zero and } \\
\frac{\Z^{F^0}}{(I-A^T_F)\Z^{F^0}} \oplus {\rm ker}(I-A^T_E) & \hbox{ in degree one. }
\end{align*}
\end{enumerate}
\end{theorem}
\begin{proof}
    By \cite[Corollary 7.7]{MR4700629}, the binary system associated to $(\xi^0, \xi^1)$ satisfies (1) and (2).  Let $\mathcal G$ denote the associated Deaconu--Renault groupoid. Then it follows from (1) that $C^*(\mathcal G)$ is a unital UCT Kirchberg algebra \cite{A-D:PI}. Moreover, $C^*(\mathcal G)$ is Morita equivalent to the  stable Ruelle algebra $R^u(X_\xi ,\sigma_\xi ,P_\xi)$ (see \cite[Section 6]{MR4700629} for the definition $R^s(X_\xi ,\sigma_\xi ,P_\xi)$, and the discussion on page 39 of that paper). It follows that they have the same $K$-theory. In particular, by \cite[Theorem 6.1]{MR4700629}, we obtain the desired $K$-theory groups. This shows (3).
\end{proof}

\begin{remark} \label{rem:embH1H2}
Suppose that $E$ and $F$ are finite directed graphs where the number of vertices of $F$ is less than or equal to the number of vertices of $E$. Let $A_E$ and $A_F$ denote the respective adjacency matrices. Then embeddings $\xi^0, \xi^1 : F \to E$ satisfying (H0)--(H3) can be constructed whenever the following holds: 
\[ 2a_{ij} +1 \leq b_{ij} \hbox{ for }1\le i,j \le N\]
where $N$ is the number of vertices in $F$, $a_{ij}$ is the $(i, j)-$entry in $A_F$ and $b_{ij}$ is the $(i,j)-$entry in $A_E$. 
\end{remark}

\begin{example} \label{ex:IanBasEx} The most basic example of the general setup is obtained by taking $F$ with adjacency matrix $[1]$ and $E$ with adjacency matrix $[3]$. In more detail, we take $F^0=\{v\}$ and $F^1=\{e\}$, and $E^0=\{w\}$ and $F=\{ f_0, f_1, f_2 \}$, see Figure \ref{fig:GraphEmbe}. Let $\xi^0 : F \rightarrow E$ and $\xi^1: F \rightarrow E$ be the embeddings given by $\xi^0(e) = f_0$ and $\xi^1(e) = f_1$. It is easy to see that $(F, E, \xi^0, \xi^1)$ satisfy (H0)--(H3).  (Indeed, up to relabelling, these are the only such embeddings for $F$ and $E$.)

\begin{figure}[ht]
    \centering
\begin{tikzpicture}[->, >=stealth, auto, thick]
\node (v2) at (2,0) {};
\fill (v2) circle[radius=2pt];
\node at (2, -0.3) {$v$};

\node (w) at (8,0) {};
\fill (w) circle[radius=2pt];
\node (w_1) at (8,-0.3) {$w$};

\path[->,in=50,out=-50,loop,scale=3] (v2) edge (v2);
\node at (3.1,0) {$e$};

\path[<-,in=230,out=-230,loop,scale=3] (w) edge (w);
\node at (6.1,0) {$f_0$};

\path[->,in=230,out=-230,loop,scale=6] (w) edge (w);
\node at (6.9,0) {$f_1$};

\path[->,in=51,out=-51,loop,scale=3] (w) edge (w);
\node at (9.1,0) {$f_2$};

\node at (7.5, 1) {$E$};
\node at (2.5, 1) {$F$};

\draw[->, dashed]
(2.8,0.5) to[out=30, in=150]
node[above] {$\xi^0$}
(6.4,0.2);

\draw[->, dashed]
(2.8,-0.5) to[out=-30, in=-150]
node[below] {$\xi^1$}
(7.1,-0.2);

\draw[->, dashed]
(2.8,-0.5) to[out=-30, in=-150]
node[below] {$\xi^1$}
(7.1,-0.2);
\end{tikzpicture}
 \caption{Graph embeddings for adjacency matrices $A_F = [1]$ into $A_E = [3]$}
    \label{fig:GraphEmbe}
    \end{figure}
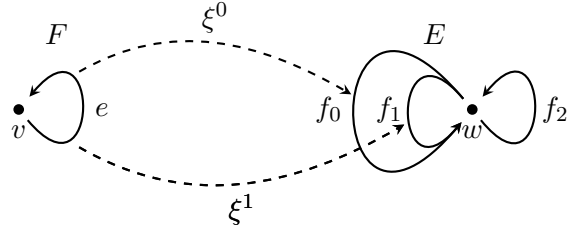

Let $\Sigma^+_E$ denote the one-sided infinite path space of $E$, that is, sequences of the form $(x_n)_{n \in \mathbb N}$ where $x_n \in E^1$ and $t(x_n) = i (x_{n+1})$ for every $n \geq 1$, and let $\sigma : \Sigma^+_E \to \Sigma^+_E$ be the left shift map.

For $x, y \in \Sigma_E^+$, let $x \sim y$ and $y \sim x$ if one of the following holds
\begin{enumerate}
    \item there is a $n \geq 1$ such that $x_n = f_0$, $x_m = f_1$ for every $m > n$, and $y_n = f_1$, $y_m = f_0$ for every $m > n$, or
    \item there is $n \geq 1$ such that $x_n = y_n = f_2$, and $x_m = f_0$, $y_m = f_1$ for every $m > n$, 
    \item $x =y$.
    \end{enumerate}
    Then 
    \[X^+_\xi := \Sigma^+ / \sim,
    \]
    and $\sigma_\xi : X^+_\xi \to X^+_\xi$ is the map satisfying $\sigma_\xi(\pi_\xi (x)) = 
    \pi_\xi(\sigma(x))$ for every $x \in \Sigma^+_E$, where $\pi_\xi : \Sigma^+_E \to X^+_\xi$ is the quotient map.

    For $k \in \mathbb{Z}_{\geq 0} \cup \{\infty\}$, let
\[ \Sigma_{E, k}^+ = \{ x \in \Sigma_E^+ \mid \exists I \subset \mathbb N, |I| = k,  \text{ such that } x_n = f_2 \iff n \in I\}.\]
In other words, $\Sigma_{E, k}^+$ consists of those sequences in $\Sigma_E^+$ where the edge $f_2$ appears exactly $k$ times. The sets $\Sigma_{E, k}^+$, $0 \leq k \leq \infty$ are pairwise disjoint \cite[Proposition 3.4 (1)]{MR4700629} and the closure of $\cup_{k=1}^\infty \Sigma_{E, k}^+$ is $\Sigma_E^+$ \cite[Proposition 3.5 (4)]{MR4700629}. Moreover, each $X_k^+$ is invariant under the equivalence relation $\sim$ \cite[Proposition 3.7 (3)]{MR4700629}.

As shown in \cite[Section 7]{MR4700629}, see in particular Corollary~7.7, $X_\xi$ is homeomorphic to a disjoint union of circles and points (with both occurring). Note in particular, that for $\pi_\xi(\Sigma_{E, 0}^+)$ the homeomorphism restricts to the map $\theta : \pi_\xi(\Sigma_{E, 0}^+) \to \mathbb T$ where 
\[ \theta(\pi_\xi(x)) :=  \exp(2\pi i \sum_{j=1}^\infty \epsilon(x_j) 2^{-j}),\]
where $\epsilon(x_j) = i \in \{0,1\}$ if $x_j \in \xi^i(F_1)$. That this map is well defined follows from \cite[Theorem 3.16 (2)]{MR4700629}. Furthermore,  
\begin{align*}
\theta(\pi_{\xi}(\sigma(x))) &= \exp(2\pi i \sum_{j=1}^\infty \epsilon(x_{j+1}) 2^{-j}) \\
&= \exp(2\pi i \sum_{j=1}^\infty \epsilon(x_j) 2^{-j+1}) = (\theta(\pi_{\xi}(x)))^2,
\end{align*}
which is to say, $\sigma_\xi|_{\pi_\xi(X_0^+)}$ is the two-fold covering map of the circle, as in Example~\ref{ex:twofold}, for the details see \cite[Section 7]{MR4700629}.

\end{example}

\begin{example} \label{Ex:O2Put}
Let $F$ and $E$ be the directed graphs with adjacency matrices 
\[ A_F =  [2] \hbox{ and } A_E = \left[ \begin{array}{cc} 5 & 3 \\ 5 & 5 \end{array} \right].
\]
Using Remark \ref{rem:embH1H2}, there are graph embeddings $\xi^0, \xi^1 : F\to E$ satisfying (H0)-(H2) and that $A_E$ is primitive (so (H3) also holds). By Theorem~\ref{thm:PutSys}, we obtain the binary system $(X_\xi, \sigma_\xi)$ associated to $(\xi^0, \xi^1)$. The associated Deaconu--Renault groupoid $\mathrm{C}^*$-algebra $C^*(\mathcal G)$ is the unital UCT Kirchberg algebra with vanishing $K$-theory (since $I-A_F$ and $I-A_E$ are both invertible over $\Z$). By the Kirchberg--Phillips classification theorem \cite{Kirch:Class, Phillips:Class}, this implies that $C^*(\mathcal G)$ is isomorphic to the Cuntz algebra $\mathcal{O}_2$.
\end{example}

\begin{example}
    Let $F$ and $E$ be the directed graphs with adjacency matrices 
\[ A_F =  [1] \hbox{ and } A_E = \left[ \begin{array}{cc} 3 & 1 \\ 1 & 0 \end{array} \right].
\]
As in the previous example, using Remark \ref{rem:embH1H2}, there are graph embeddings satisfying (H0)-(H2) and that $A_E$ is primitive so that (H3) holds. Given such a pair of embeddings $\xi : F \to E$, by Theorem~\ref{thm:PutSys}, we obtain the binary system $(X_\xi, \sigma_\xi)$ associated to $(\xi_1, \xi_2)$. As in Example \ref{ex:OnePreSFT}, there is a point in $X_\xi$ whose preimage set is exactly one point.
\end{example}

\section{orbit-breaking in Deaconu--Renault groupoids} \label{sec:OBDR}
Let $\varphi : X \to X$ be a surjective local homeomorphism on a compact metric space $X$. In particular, for the present section we do {\bf not} assume that $(X, \varphi)$ is locally expanding. Denote by $\mathcal G$ the associated Deaconu--Renault groupoid. In this section, we apply the results of the second author \cite{Put:K-theoryGroupoids} to our particular situation. 

In more detail, we construct a subgroupoid by ``breaking the orbit'' of $\varphi$ at a closed subset $Y\subset X$ with the aim of producing simple groupoid $\mathrm{C}^*$-algebras. This will generalize the orbit-breaking groupoids and their $\mathrm{C}^*$-algebras constructed from a minimal homeomorphism together with a closed  subset meeting every orbit at most once, see for example \cite{Putnam:MinHomCantor, DPS:DynZ, DPSmain}. Since the orbit structure of $(X, \varphi)$ is more complicated than orbits associated to  homeomorphisms, we require additional assumptions on $Y$. Recall that $X_{\mathrm{ep}}$ denotes the set of eventually periodic points of $(X, \varphi)$.

Let $Y \subseteq X$ be a closed subspace such that
\begin{enumerate}
\item[(i)] $Y_{\mathrm{ep}} := X_{\mathrm{ep}} \cap Y$ is empty or contains exactly one fixed point, and
\item[(ii)] $Y$ intersects each (generalized) orbit at most once, that is, for each $x\in X$, $\orb(x)\cap Y$ is the empty set or a singleton.
\end{enumerate}

\begin{remark}
    These two assumptions will be the standing assumptions on $Y$. There are two cases. The case when $Y_{\mathrm{ep}} := X_{\mathrm{ep}} \cap Y$ is empty is used in Section \ref{sec:OB} and the case when $Y_{\mathrm{ep}} := X_{\mathrm{ep}} \cap Y$ is a single fixed point is used in Section \ref{sec:OBO2LC}. The reader might find it useful to consider only one of the two cases (likely the case when $Y_{\mathrm{ep}}$ is empty) on a first read of the present section.
\end{remark}

\begin{lemma} \label{lem:GoodY}
Suppose that $X$ is a compact metric space, $\varphi$ is a surjective, local homeomorphism, and $Y \subseteq X$. The following are equivalent:
\begin{enumerate}
\item $Y$ meets every (generalized) orbit at most once and contains no eventually periodic points.
\item For $x, y \in Y$ and $k, l \ge 0$ with $\varphi^k(x)=\varphi^l(y)$ we have that $x=y$ and $k=l$.
\end{enumerate}
\end{lemma}
\begin{proof}
First we show that (1) implies (2). Assume that $\varphi^k(x)=\varphi^l(y)$ for $x, y \in Y$ and $k, l \ge 0$. Since $\varphi^k(x)=\varphi^l(y)$, both $x$ and $y$ are in the same generalized orbit. Since  $x, y \in Y$ and $Y$ meets every generalized orbit at most once, we must have $x=y$. Hence, $\varphi^k(x)=\varphi^l(x)$ and hence $k=l$ since $Y$ contains no eventually periodic points.

Next, we show that the (2) implies (1). 

Firstly, (2) implies that $Y$ contains no eventually periodic points; otherwise there would be $y\in Y$ and $N, K \in \N$ such that $\varphi^{N+K}(y)=\varphi^N(y)$, which it contradicted by (2). 

Secondly, (2) implies $Y$ meets every (generalized) orbit at most once. Suppose $x, y$ are in $Y$ and $\orb(x)=\orb(y)$. By the definition of (generalized) orbit there exists $k, l \ge 0$ with $\varphi^k(x)=\varphi^l(y)$. By (2), $x=y$ as required. 
\end{proof}

\begin{remark}
Lemma~\ref{lem:GoodY} (2) implies that for each $k>0$, $\varphi^k|_Y$ is injective and that $\varphi^k(Y) \cap \varphi^l(Y) = \emptyset$ for $k, l \geq 0$ and $k\neq l$.

Moreover, condition (2) can be stated in terms of the Deaconu--Renault groupoid. It is equivalent to 
\[ \{ (x, p, z) \mid (x, p, z) \in \mathcal{G} \hbox{ and }x, z\in Y\} \subseteq \mathcal{G}^{(0)}. \]
\end{remark}

\begin{definition} Let $Y \subseteq X$ be a closed subset satisfying (i) and (ii) as above, and let $U = X \setminus Y_\mathrm{ep}$.  We define the subgroupoid $\mathcal{G}_Y \subset \mathcal G|_U$  as follows. Let $(x, p, z) \in \mathcal{G}|_U$. Then  $(x, p, z) \in \mathcal{G}_Y$ if one of the following holds:
\begin{enumerate}
\item[(I)] $\orb(x)\cap Y \subseteq Y_{\mathrm{ep}}$,
\item[(II)] $\orb(x)\cap Y \setminus Y_{\mathrm{ep}} =\{y\}$ for some $y\in Y$ and $x, z \in \orb^-(y)$,
\item[(III)] $\orb(x)\cap Y \setminus Y_{\mathrm{ep}} =\{y\}$ for some $y\in Y$ and $x, z \in \orb(y) \setminus \orb^-(y)$.
\end{enumerate}
\end{definition}

Here $\mathcal G|_U$ denotes the reduction to $U$, which is open since $Y_{\mathrm{ep}}$ is empty or a single point. Also, note that if $(x, p, z) \in \mathcal{G}|_U$, then $\orb(x)=\orb(z)$, so the conditions above are symmetric in $x$ and $z$.

In the case when $Y_{\mathrm{ep}}$ is the empty set, we have that $\mathcal G|_U = \mathcal G$ and the three conditions above become
\begin{enumerate}
\item $\orb(x)\cap Y$ is the empty set,
\item $\orb(x)\cap Y =\{y\}$ for some $y\in Y$ and $x, z \in \orb^-(y)$,
\item $\orb(x)\cap Y =\{y\}$ for some $y\in Y$ and $x, z \in \orb(y) \setminus \orb^-(y)$,
\end{enumerate}
where $(x, p, z) \in \mathcal{G}$.

With $Y$ and $U$ as above,  define
\[ L := \{ (x, n , z) \in \mathcal{G} \mid x \in \orb^{-}(Y) \hbox{ and }z \not\in \orb^{-}(Y) \},
\]
where we note that
\begin{enumerate}
\item $\orb^{-}(Y)=\{ w \in X \mid \varphi^k(w) \in Y \hbox{ for some }k \geq 0 \}$ and
\item $(x, n, z) \in \mathcal{G}$ implies that $\orb(x)=\orb(z)$.
\end{enumerate}
Note also that if $Y_{\mathrm{ep}}= \{w\}$ where $\{w\}$ is a fixed point, then $\orb(w)=\orb^{-}(w)$. Hence, if $(x, n, y) \in L$, then $x, y \neq w$ since $w \notin \orb(Y) \setminus \orb^-(Y)$. It follows that $L \subset \mathcal G|_U$.

The next lemma tells us that, with $L$ as above, $L$ and $\mathcal G_Y \subset \mathcal G|_U$ satisfy the \emph{Situation 2. Subgroupoids} criteria of \cite[Section 2]{Put:K-theoryGroupoids}. This will allow us to apply \cite[Theorem  2.4]{Put:K-theoryGroupoids} to obtain Theorem~\ref{thm:BreExaSeq}.

\begin{lemma} \label{lem:L}
Suppose that $X$ is a compact metric space, $\varphi$ is a surjective, local homeomorphism,  $Y \subseteq X$ is a closed subset that meets every (generalized) orbit at most once, and $Y_{\mathrm{ep}}$ is empty or contains exactly one fixed point. Then with $L$ defined as in the previous paragraph and $U = X \setminus Y_{\mathrm{ep}}$, 
\begin{enumerate}
\item if $(x, n, z) \in L$, then there exists $k\ge 1$ and $l\ge 0$ such that
\begin{enumerate}
    \item[(a)]$n=k-l$,
    \item[(b)]$\varphi^k(x)=\varphi^l(z)$,
    \item[(c)]$\varphi^{k-i}(x)\ne \varphi^{l-i}(z)$ for $1\le i \le {\rm min}\{k, l\}$,
    \item[(d)]there is unique $i\in \{ 0, \ldots, k-1 \}$ such that $\varphi^i(x) \in Y$.
\end{enumerate}
\item $L \cap L^{-1}=\emptyset$,
\item $L$ is closed in $\mathcal G|_U$,
\item $\mathcal G_Y = \mathcal G|_U - L - L^{-1}$,
\item $L \mathcal G_Y , \mathcal G_Y L \subset L$.
\end{enumerate}
\end{lemma}

\begin{proof}
For (1), let $(x,n,z) \in L$. Then, by the definition of $L$, there exists $k\in \N$ such that there exists $l\ge 0$ with $\varphi^k(x)=\varphi^l(z)$ and $n=k-l$. Items (a) and (b) hold for any $k$ with this property. Next, take the smallest $k \in \N$ with this property. Then  (c) holds. Finally, since $x \in \orb^-(Y)$, there is $i \geq 0$ and $y \in Y$ such that $\varphi^i(x) = y\in Y$. However, $z \in \orb(y) \setminus \orb^{-}(y)$, so $\varphi^l(z)=\varphi^k(x)$, $\varphi^{l+1}(z)=\varphi^{k+1}(x), \dots$, are not equal to $y$. Hence, $i \in \{ 0, \ldots , k-1\}$. Furthermore $i$ is unique because $Y$ meets every (generalized) orbit at most once. Hence, (d) holds. This completes the proof of (1).

That (2) holds is clear from the definition of $L$ and the fact that every orbit meets $Y$ at most once.

To show that $L$ is closed in $\mathcal G|_U$, let $((x_i, n_i - m_i, y_i))_{i\in \mathbb N} \subset L$ and suppose that $(x_i, n_i - m_i, y_i) \to (x,n-m, y) \in \mathcal G|_U$ as $i \to \infty$. Without loss of generality, we may assume that $n_i - m_i = n - m$ for every $i\in \mathbb N$, that $x_i \to x$, and that $y_i \to y$ as $i \to \infty$. Let $U, V \subset X$ be open subsets containing $x$ and $z$ respectively, with $\varphi^n(U) = \varphi^m(V)$, and such that $\varphi^n|_U$ and $\varphi^m|_V$ are homeomorphisms onto their shared image. Then 
\[ W(U,V) := \{ (\hat{x}, n-m, \hat{z}) \in \mathcal G \mid \hat{x} \in U, \hat{y} \in V\}\]
is open in $\mathcal G$. In particular, there is $N >0$ such that $(x_i, n-m, y_i) \in W(U,V)$ for every $i \geq N$. Now $x_i \in \orb^-(Y)$ for every $i \in \mathbb N$, so there exists $K_i \in \mathbb{Z}_{\geq 0}$ and $y_i \in Y$ such that $\varphi^{K_i}(x_i) = y_i$ for every $i \in \mathbb N$. Since $Y$ meets every orbit at most once, and by definition of $L$, we have $z_i \in \orb(y_i) \setminus \orb^-(y_i)$ and so we must have $0 \leq K_i < n$.

In particular, there are only finitely many possible values of $K_i$, so by the pigeonhole principle, there is $K$, $0 \leq K < n$ such that $K_i = K$ for infinitely many $i$. Thus passing to a subsequence if necessary, we have $\varphi^K(x_i) = y_i$ for every $i$. Then $\varphi^K_i(x_i) \to \varphi^K(x)$ and $\varphi^K(x) \in Y$ since $Y$ is closed. It follows that $x \in \orb^-(Y)$. Also, $\varphi^m(x_i) = \varphi^n(z_i)$ and $z_i \to z$, we get $\varphi^m(x) = \varphi^n(z)$ and $z \in \orb(y) \setminus \orb^{-}(y)$ since $K < n$. Thus $(x, n-m, z) \in L$, which shows (3).

For (4), note that \[
L^{-1} = \{ (y, m-n, x) \mid x \in \orb^-(Y), z \notin \orb^-(Y)\}.
\]
Thus if $(x, m-n, z) \in \mathcal G - L - L^{-1}$ either $x, z \in \orb^-(Y) \setminus Y_{\mathrm{ep}}$,  or $(x, n-m, z) \notin \orb^-(Y)$. If $x, z \in \orb^-(Y) \setminus Y_{\mathrm{ep}}$ then $(x, n-m, y) \in \mathcal G_Y)$. In the latter case, if $x \in \orb(Y)$, then so too is $z$, so that $x, z \in \orb(Y) \setminus \orb^-(Y)$ whence $(x, n-m, z) \in \mathcal G_Y$. Otherwise, $Y$ does not meet the orbit of $x$ or $y$, in which case $(x, n-m, z) \in \mathcal G_Y$. Hence $\mathcal G - L - L^{-1} \subset \mathcal G_Y$. On the other hand, if $(x, n-m, z) \in \mathcal G_Y$, then either $x, z \not \in \orb(Y)$, in which case $(x, n-m, y) \notin L \cup L^{-1}$. If $x, z \in \orb(Y)$ then either $x, z \in \orb^{-}(Y)$, in which case $(x, n-m, z) \notin L \cup L^{-1}$, or $x, z \notin \orb^-(Y)$, in which case $(x, n-m, z) \notin L \cup L^{-1}$. Hence $\mathcal G_Y = \mathcal G - L - L^{-1}$, showing (4).

Finally, to show (5), suppose that $(x, n-m, z) \in L$, $(x', n'-m', z') \in \mathcal G_Y$ and 
\[
(x, n-m, z) (x', n'-'m, z') = (x, (n+n') - (m+m'), z').
\]
Since $(x, n-m, z) \in L$, we have $x \in \orb^-(y)$ for some $y \in Y$, and $z \in \orb(y) \setminus \orb^-(y)$. It follows that $x' = z \in \orb(y) \setminus \orb^-(y)$. Since $(x', n'-m', z') \in \mathcal G_Y$, this forces $z' \in \orb(y) \setminus \orb^-(y)$. Hence $(x, (n+n') - (m+m'), z') \in L$, and $L \mathcal G_Y \subset L$. The proof that $\mathcal G_Y L \subset L$ is similar.
\end{proof}

\begin{proposition} \label{prop:opensub}
The set $\mathcal{G}_Y$ is an open subgroupoid of $\mathcal{G}|_U$.
\end{proposition}

\begin{proof}
    This follows directly from (3) and (4) in the previous theorem.
\end{proof}

\begin{corollary} \label{cor:EtaAndAm}
$\mathcal{G}_Y$ is \'etale and amenable.
\end{corollary}
\begin{proof}
Since $\mathcal G$ is \'etale and amenable, so is $\mathcal G|_U$ by Proposition~\ref{prop:GUopen}. As in the proof of Proposition~\ref{prop:GUopen},  that $\mathcal{G}_Y$ is \'etale and amenable then follow from the fact $\mathcal{G}_Y$ is an open subgroupoid of the amenable \'etale groupoid $\mathcal G|_U$. 
\end{proof}

The following theorem will allow us to produce the inclusions of  Theorem~\ref{thm:inclusions}.

\begin{theorem} \label{thm:subsetOrbitBreak} 
Let $w \in X$ be a fixed point. Suppose that $Y_1 \subset Y_2$ are closed, nested subsets of $X$ such that one of the following holds. 
\begin{enumerate}
    \item $Y_2 \cap X_{ep}=\emptyset$ or
    \item $Y_1\cap X_{ep}=Y_2\cap X_{ep}=\{ w \}$.
\end{enumerate} Suppose further that $Y_2$ (and hence $Y_1$) meets every (generalized) orbit at most once. Then $\mathcal{G}_{Y_2}$ is an open subgroupoid of $\mathcal{G}_{Y_1}$. In particular, $C^*_r( \mathcal{G}_{Y_2})$ is a $\mathrm{C}^*$-subalgebra of $C^*_r( \mathcal{G}_{Y_1})$.
\end{theorem}

\begin{proof}
  That $\mathcal G_{Y_2}$ is a subgroupoid of $\mathcal G_{Y_1}$ is clear.  By Proposition~\ref{prop:opensub}, $\mathcal G_{Y_1}$ and $\mathcal G_{Y_2}$ are both open in $\mathcal G$. Thus since $\G_{Y_2} \subset \mathcal G_{Y_1}$, we have that $\mathcal G_{Y_2}$ is open in $\mathcal G_{Y_1}$.
\end{proof}

With $L$ as above, we have
\[ L^{-1} L = \{ (x, k-l, z) \in \mathcal G \mid x, z \in \orb^-(Y) \setminus Y_{\mathrm{ep}} \}, \]
and 
\[ LL^{-1} = \{(x, k-l z) \in \mathcal G \mid x, z \in \orb(Y) \setminus \orb^-(Y)  \}. \]
\begin{definition}
Suppose $y\in Y$, $x \in \orb^{-}(y)$, $k$ is the unique natural number such that $\varphi^k(x) = y \in Y$, and $V \subseteq X$ is an open set such that $x\in V$ and $\varphi^{k}|_V$ is injective. Define
\[ W(x, V) := \{ (v, 0, v) \mid v \in v \cap \varphi^{-k}(Y) \}.
\]
\end{definition}

\begin{definition}
Suppose $y\in Y$, $x \in \orb(y) \setminus \orb^{-}(y)$, $k$ is the unique element in the natural numbers such that $\varphi^{k+l}(x) = \varphi^l(y) \in Y$, and $V_1, V_2 \subseteq X$ are open sets such that $x\in V_1$,  $y\in V_2$, $\varphi^{k}|_{V_1}$ is injective, $\varphi^{l}|_{V_2}$ is injective, and $\varphi^{k+l}(V_1)=\varphi^l(V_2)$. Define
\[ Z(x, V_1) := \{ (v, 0, v) \mid v \in V_1 \cap \varphi^{-k-l}(Y) \}.
\]
\end{definition}

\begin{lemma} \label{lem:rsLtop}
Let $L = \{(x, n, z) \in G \mid x \in \orb^-(Y), z \notin  \orb^-(Y)\}$. Then the following hold. 
\begin{enumerate}
\item The range of $L$ is $r(L)= \{ (x, 0, x) \mid x \in \orb^{-}(Y) \setminus Y_{\mathrm{ep}}\}$, and
\[ \{ W(x, V) \mid \exists k >0: \varphi^k(x)\in Y, x \in V \text{ open }, \varphi^k|_V \text{ injective}\}
\]
forms a base for the quotient topology associated to the map $r|_L$. 
\item the range of $L^{-1}$ is $r(L^{-1})= \{ (x, 0, x) \mid x \in \orb(Y) \setminus \orb^{-}(Y) \}$, and
\begin{align*}
\{ Z(x, V_1) \mid \exists k, l \geq 0: &\varphi^k(x) \in \varphi^l(Y), x \in V_1 \text{ open}, y \in V_2 \text{ open}, \\ &\varphi^k|_{V_1},  \varphi^l|_{V_2} \text{ injective and } \varphi^k(V_1) = \varphi^l(V_2)\}.
\end{align*}
forms a base for the quotient topology associated to the map $r|_{L^{-1}}$. 
\end{enumerate}
\end{lemma}

\begin{proof} We show (1). The proof of (2) is similar and so omitted. That $r(L) = \{(x, 0, x) \mid x \in \orb^-(Y) \setminus Y_{\mathrm{ep}}\}$ is clear.  Let $W(x,V)$ be the set associated to $y \in Y$, $x \in \orb^-(Y) Y_{\mathrm{ep}}$,  $k$ the unique natural number satisfying $\varphi^k(x) = y$, and $V$ an open set containing $x$ such that $\varphi^k|_V$ is injective.  Then
\[
r|_L^{-1}(W(x,V)) = \bigcup_{\hat{x} \in U \cap \varphi^{-k}(Y) \setminus Y_{\mathrm{ep}}} \bigcup_{\hat{z} \in \operatorname{orb}(\hat{x})} W(\hat{x}, \hat{z}, U_{\hat{x}}, V_{\hat{z}}) \cap L,
\]
where $U_{\hat{x}}$ and $V_{\hat{z}}$ are open sets containing $\hat{x}$ and $\hat{z}$ respectively, such that $\varphi^m|_{U_{\hat{x}}}$ and $\varphi^n|_{V_{\hat{z}}}$ are homeomorphisms onto their shared image. Since the set $W(\hat{x}, \hat{z}, U_{\hat{x}}, V_{\hat{z}})$ is a basic set, its intersection with $L$ is open in $L$. Thus  $W(x,V)$ is open in the quotient topology. 

That the $W(x,V)$ cover all of $r(L)$ is clear. Given two sets $W(x_1,V_1)$ and $W(x_2,V_2)$, suppose there is $r((x_3, n-m, z)) \in W(x_1,V_1)\cap W(x_2,V_2)$. Then $x_3 \in V_3 = V_1 \cap V_2$ so we have $r((x_3, n-m, z)) \in W(x_3,V_1 \cap V_2)$. Since $\varphi^{k_1}|_{V_1}$ and $\varphi^{k_2}(V_2)$ are both injective, uniqueness forces and $k_1 = k_2 = k$ and it is immediate that $\varphi^{k}|_(V_1 \cap V_2)$ and that $\varphi^k(x_3) \in Y$. Thus the $W(x, V)$ are a basis for the quotient topology. 
\end{proof}

Let 
\[ \mathcal H = \mathcal G|_{{r(L)} \cup s(L)}, \quad \mathcal G' = \mathcal G|_U \setminus L \cup L^{-1}, \quad \mathcal H' = \mathcal G'|_{r(L) \cup s(L)}.
\]

Let $(x, k-l, z) \in  L L^{-1} \subset \mathcal H$. Then 
\[
(x, k-l, z) = (x_1, k_1 - l_1, z_1)(x_2, k_2 - l_2, z_2)^{-1} =  (x_1, l_1 + l_2 - k_2 - k_2, x_2),
\] where $(x_1,n_1, z_1), (x_2, n_2, z_2) \in L$. 

Let 
\[V(x_1, z_1) = \{ (\hat{x}_1, n, \hat{z_1}) \mid \hat{x_1} \in U_{x_1},\hat{z_1} \in U_{z_1}\}, \]
and
\[ V(z_2, x_2) = \{ (\hat{z}, n, \hat{x}) \mid \hat{z} \in U_{z_2},\hat{x_2} \in U_{x_2}\} \]
where $U_{\hat{x_i}}, U_{\hat{z_i}}\subset X$ are open neighbourhoods of $x_i$, $z_i$, respectively, such that $\varphi^{k_i}(U_{x_i}) = \varphi^{l_i}(U_{z_i})$ and $\varphi^{k_i}|_{U_{x_i}}, \varphi^{l_i}|_{U_{z_i}}$ are homeomorphisms onto $\varphi^{k_i}({U_i}) = \varphi^{z_i}(U_{z_i})$.

Then \[ \left(V(z_1, x_1) \cup V(x_2, z_2)\right)L\left(V(z_1, x_1) \cup V(x_2, z_2)\right)L^{-1}\] is an open set in $\mathcal H$ containing $(x, k-l, z)$.

Recall that an \emph{abstract transversal} in a topological groupoid $\mathcal G$ is a closed subset $N \subset \mathcal G^{(0)}$ which meets every orbit at least once. Let $\mathcal G|_N := \{ \gamma \in \mathcal G \mid r(\gamma), s(\gamma) \in N\}$. If $N$ is an abstract transversal in $\mathcal G$ such that the restriction of the range and source to $\{ \gamma \in \mathcal G \mid s(\gamma) \in Y\}$ are open maps, then $N$ is a $(\mathcal G, \mathcal G|_N)$-equivalence in the sense of \cite[Definition 2.1]{MuhRenWil:Eqiv} (see Example 2.7 of that paper), and the $\mathrm{C}^*$-algebras $C^*(\mathcal G)$ and $C^*(\mathcal G|_N)$ are Morita equivalent \cite[Theorem 2.8]{MuhRenWil:Eqiv}.

\begin{proposition} \label{prop:transversal}
Let $\mathcal H$
be as above. Then
\begin{enumerate}
\item $Y \setminus Y_{\mathrm{ep}} \subseteq \mathcal{H}^{(0)}$ is an abstract transversal in $\mathcal{H}$,

\item  $C^*(\mathcal{H})$ is Morita equivalent to $C_0(Y \setminus Y_{\mathrm{ep}})$.
\end{enumerate}
\end{proposition}
\begin{proof}
To show that $Y$ is an abstract transversal in $\mathcal H$ we need to verify that $Y$ is closed in $\mathcal H^{(0)}$ and meets each orbit at least once.

Since $\mathcal H = L^{-1}L \sqcup LL^{-1} \sqcup  L \sqcup L^{-1}$, we have
\[
\mathcal H = \{ (x, n , z) \in \mathcal G \mid x, y \in \orb(y) \text{ for some } y \in Y \setminus Y_{\mathrm{ep}}\}.
\]
Hence $Y$ meets every $\mathcal H$ orbit, showing (1).

To verify (2), we furthermore show that the restrictions of the range and source maps to be open when restricted to $\{ \gamma \in \mathcal H \mid s(\gamma) \in Y  \}$. That $C^*(\mathcal H)$ and $C(Y)$ are Morita equivalent then follows from  \cite[Theorem 2.8]{MuhRenWil:Eqiv}, as observed above. 

We have
\[ L \cap \{ (x, k-l, y) \in \mathcal G \mid y \in Y \} = \emptyset,\]
while
\begin{align*}
L^{-1}& \cap \{(x, k-l, y) \in \mathcal G \mid y \in Y \} \\
 &= \{(x, k-l, y) \in \mathcal G \mid x \in \orb(y) \setminus \orb^{-}(y), y \in Y\}.
\end{align*} 

By Lemma~\ref{lem:rsLtop}, both $r|_{s^{-1}(Y)}(L^{-1})$ and $s|_{s^{-1}(Y)}(L^{-1})$ are open.

Now we consider the intersection with $L^{-1}L, LL^{-} \subset \mathcal H$
Note that $L^{-1} L$ contains points of the form $(x, n, z)$ where both $x, z \notin \orb^-(Y)$. In particular, \[
L^{-1}L \cap {\{ (x, n , y)  \in \mathcal H \mid y \in Y \setminus Y_{\mathrm{ep}}\}} = \emptyset.
\]

On the other hand, $LL^{-1}$ contains points of the form $(x, n, z)$ where both $x, z \in \orb^{-}(y)$. Hence 
\[
LL^{-1}\cap {\{ (x, n , y)  \in \mathcal H \mid y \in Y \}}  = \{ (y, 0, y) \mid y \in Y \setminus Y_{\mathrm{ep}}\}.
\]

Let $(y,0,y) \in  LL^{-1} \cap \{(x,n,y) \in  \mathcal H \mid y \in Y \setminus Y_{\mathrm{ep}} \}$ be given by
\[(y,0,y) = (x_1, k_1 - l_1, z_1)(z_2, -k_2 + l_2, x_2),
\]
where $x_i \in \orb^{-}(Y)$ and $z_i \in \orb(Y) \setminus \orb^-(Y)$. Then $z_1 = z_2$, $x_1 = x_2 = y$ and $k_1 - l_1= k_2 - l_2$. Let $V_{x_1}$ be an open neighbourhood of $y$ and $V_{z_1}$ an open neighbourhood around $z_1$ such that $\varphi^{k_1}(V_{x_1}) = \varphi^{l_1}(V_{z_1)}$ and $\varphi^{k_1}|_{U_{x_1}}$ and $\varphi|_{V_{z_1}}$ are homeomorphisms onto their image. Let
\[ V(y, z_1) = \{ (\hat{x}, k_1 - l_1, \hat{z}) \mid \hat{y} \in U_{x_1}, \hat{z} \in U_{z_1} \}.\]
Similarly, find $V_{x_2}$ and $V_{z_2}$ for $x_2$ and $z_2$ and let 
\[ V(y, z_2) = \{ \{ (\hat{x}, k_2 - l_2, \hat{z}) \mid \hat{y} \in V_{x_2}, \hat{z} \in V_{z_2} \}.
\]
let $V := V(y, z_1) \cap V(y, z_2)$. Then $V$ is an $(y, 0,y)$ in $\mathcal G$. We have
\[ VV^{-1} \subset U(z, z_1) U(y, z_2)^{-1},\]
so $VV^{-1}$ is an open neighbourhood of $(y, 0, y)$ in $LL^{-1}$ and hence in $\mathcal H$.  Moreover,
\[
VV^{-1} = \{(\hat{y}, 0, \hat{y}) \mid \hat{y} \in V_{x_1} \cap V_{x_2} \cap Y \},
\]
since $(\hat{y}, n, \hat{z})(\tilde{y}, n, \tilde{z})^{-1}$ implies that $\hat{z} = \tilde{z}$, and we must have $\hat{y} = \tilde{y}$ since $Y$ meets every orbit at most once.

Then $s(VV^{-1})$ and $r(VV^{-1})$ are open by Lemma~\ref{lem:rsLtop}.

This shows (2).
\end{proof}

Now we can establish Theorem~\ref{thm:BreExaSeq}, which is the main tool for calculating the $K$-theory of the $\mathrm{C}^*$-algebras constructed in Section~\ref{sec:OBO2}.

\begin{theorem} \label{thm:BreExaSeq}
Let $X$ be a metric space, $\varphi$ a surjective, local homeomorphism, and $Y \subseteq X$ be a closed subset in $X$. Suppose that $Y$ meets every (generalized) orbit at most once and $Y_{\mathrm{ep}}$ is empty or contains exactly one fixed point. Then there is a six-term exact sequence
\begin{displaymath} 
\xymatrix{ K^0(Y \setminus Y_{\mathrm{ep}}) \ar[r] & K_0(C^*(\mathcal{G}_Y)) \ar[r]^-{\iota_*} & K_0(C^*(\mathcal{G})) \ar[d]^{\partial_{\mathrm OB}}\\
K_1(C^*(\mathcal{G})) \ar[u]^{\partial_{\mathrm OB}} & K_1(C^*(\mathcal{G}_Y)) \ar[l]_-{\iota_*} & K^1(Y \setminus Y_{\mathrm{ep}}), \ar[l]}
\end{displaymath}
where $\mathcal{G}_Y$ is the orbit-breaking groupoid obtained from $Y$, $\mathcal{G}$ is the Deaconu--Renault groupoid associated to $(X, \varphi)$, and the map $\iota_*$ on $K$-theory is induced from the open inclusion $ \iota: \mathcal{G}_Y \hookrightarrow \mathcal{G}$.
\end{theorem}
\begin{proof}  If $Y_{\mathrm{ep}}$ is non-empty, then $C^*(\mathcal G|_{X \setminus Y_{\mathrm{ep}}})$ is a hereditary $\mathrm{C}^*$-subalgebra (Proposition~\ref{prop:GUopen}), which is full since $C^*(\mathcal G)$ is simple. It follows that 
\[ K_*(C^*(\mathcal G|_{X \setminus Y_{\mathrm{ep}}})) \cong K_*(C^*(\mathcal G)).
\]
Thus in both the case that $Y_{\mathrm{ep}}$ is empty and in the case that it is a single fixed point, Lemma~\ref{lem:L} and \cite[Theorem 2.4]{MR1194074} give us a six-term exact sequence 
\begin{displaymath} 
\xymatrix{ K_0(C^*(\mathcal H)) \ar[r] & K_0(C^*(\mathcal{G}_Y)) \ar[r]^-{\iota_*} & K_0(C^*(\mathcal{G})) \ar[d]^{\partial_{\mathrm OB}}\\
K_1(C^*(\mathcal{G}) \ar[u]^{\partial_{\mathrm OB}} & K_1(C^*(\mathcal{G}_Y)) \ar[l]_-{\iota_*} & K_1(C^*(\mathcal H) ), \ar[l]}
\end{displaymath}
where $\iota : C^*(\mathcal G_Y) \to C^*(\mathcal G)$ is the inclusion map.
By Proposition~\ref{prop:transversal} (2), $C^*(\mathcal H)$ and $C_0(Y \setminus Y_{\mathrm{ep}})$ are Morita equivalent, so that \[K_*(C^*(\mathcal H)) \cong K^*(Y \setminus Y_{\mathrm{ep}}).\] The result follows.
\end{proof}

\begin{example}
    When the map $\varphi : X \to X$ is a homeomorphism, the Deaconu--Renault groupoid is the transformation  groupoid $\mathbb Z \times X$ associated to the integer action induced by $\varphi$. Based on this, the definition of orbit-breaking in the present paper is a true generalization of the original orbit-breaking construction, introduced in \cite{Putnam:MinHomCantor, Put:Excision}. 
\end{example}

\section{Locally contracting groupoids and purely infinite $\mathrm{C}^*$-algebras} \label{sec:lcgpi}
As the aim of the paper is to construct orbit-breaking algebras which are simple and purely infinite, restrictions on the system $(X, \varphi)$ (along with the restrictions on the subset $Y \subseteq X$ already discussed) will be required. To pinpoint the dynamical requirements, we follow \cite[Section 4]{A-D:PI}. A number of definitions and preliminaries are required before the main result, Proposition \ref{purelyInfProp}. This result, in both statement and proof,  is based on results in \cite{A-D:PI}, see in particular Proposition 4.2 of that paper.

\begin{proposition} \label{IsotropyEventuallyPer}
Suppose that $\mathcal{G}$ is the Deaconu-Renault groupoid associated to $(X, \varphi)$. Then, an element $x\in X$ has non-trivial isotropy in $\mathcal{G}$ if and only if $x$ is eventually periodic.
\end{proposition}
\begin{proof}
An element $x \in X$ has non-trivial isotropy if and only if there are $k, l \geq 0$ such that $(x, k - l, x) \in \mathcal G$ if and only if $\varphi^k(x) = \varphi^l(x)$ if and only if $x$ is eventually periodic.
\end{proof}

Recall from Definition \ref{def:Gset} that,  given a groupoid $\mathcal G$ and a $\mathcal G$-set $S$, we have an associated homeomorphism from $r(S)$ to $s(S)$, which we denote by $\alpha_{S}$.

\begin{definition}[{\cite[Definition 2.1]{A-D:PI}}] 
Suppose $\mathcal{G}$ is a second countable locally compact \'etale groupoid. We say that $\mathcal{G}$ is \emph{locally contracting} if, for every non-empty open set $U \subseteq \mathcal{G}^{(0)}$, there exists an open subset $V \subseteq U$ and a $\mathcal{G}$-set $S$, such that $\overline{V} \subseteq s(S)$ and $\alpha_{S^{-1}}(\overline{V})\subseteq V$ but $\alpha_{S^{-1}}(\overline{V}) \neq V$.
\end{definition}

\begin{proposition} \label{purelyInfProp}
Let $X$ be a compact metric space with no isolated points and let $\varphi: X\rightarrow X$ be an mixing, open, and locally expanding map. Suppose that $Y\subseteq X$ is a closed subset such that
\begin{enumerate}
\item for each $x\in X$, $\orb(x)\cap Y$ is the empty set or a singleton, and
\item $Y$ contains no eventually periodic points (i.e., $Y_{ep}=Y\cap X_{ep}=\emptyset$).
\end{enumerate}
Furthermore, suppose that the set of eventually periodic points of $\varphi$ is dense, but has empty interior. Let $\mathcal G_Y$ denote the associated orbit-breaking groupoid (which is a subgroupoid of the Deaconu-Renualt groupoid, $\mathcal{G}$). Then $C^*(\mathcal{G}_Y)$ is simple and purely infinite. 
\end{proposition}

\begin{proof}
To show that $C^*(\mathcal{G}_Y)$ is simple, we show that groupoid is minimal, that is, for every $x \in X$  the set $r(s^{-1}(x))$ is dense in $X$.

There are two cases. First, suppose that $r(s^{-1}(x))\cap Y = \emptyset$. Then $r(s^{-1}(x)) = \orb(x)$, so $r(s^{-1}(x))$ orbit is dense because $\orb^-(x)$ is dense by Lemma \ref{lem:denseBack}. 

Otherwise, there exists $y\in Y$ such that $y \in r(s^{-1}(x))$. If $x \in \orb^-(y)$, then again the result follows from Lemma \ref{lem:denseBack}. Otherwise, we have that $x \in \orb(y) \setminus \orb^{-1}(y)$. By Theorem \ref{lem:Preimages}, there exists $N\in \N$ such that $\varphi^{-N}(\varphi^N(x))$ contains more than one element. Using the fact that $\orb(y)$ contains not eventually periodic points, there is $w\in X$ such that 
\[ \orb^-(w)\subset \orb(y) \setminus \orb^{-1}(y)
\] 
(Note that the required point is any $w \neq x$ such that $\varphi^N(w)=\varphi^N(x)$.) The backward orbit of $\orb^-(w)$ is dense by Lemma \ref{lem:denseBack} and $\orb^-(w) \subset r(s^{-1}(x))$. Thus $\mathcal G_Y$ is minimal.  

By Proposition~\ref{IsotropyEventuallyPer} a point $x \in X$ is eventually periodic if and only if it has non-trivial isotropy. Moreover, the orbit of $x$ cannot meet $Y$ since $Y$ does not contain eventually periodic points. Hence $x$ has also has non-trvial isotropy in $\mathcal G_Y$ if and only if $x$ is eventually periodic. By assumption the set of eventually periodic points has empty interior. Thus $\mathcal G_Y$ is essentially principal. Thus $C^*(\mathcal G_Y)$ is simple \cite[Theorem 5.1]{BrnClkFarSims2014}.

The proof that $C^*(\mathcal G_Y)$ is purely infinite is similar to that of \cite[Proposition 4.2]{A-D:PI}. First, as mentioned in the previous paragraph, $\mathcal G_Y$ is essentially principal (which is the same as being essentially free in the terminology of \cite{A-D:PI}). 

Next, we show that $\mathcal G_Y$ is locally contracting. Since $\varphi$ is locally expanding, there are constants $\delta_X>0$ and $\lambda_X>1$ such that
\[ 
\lambda_X d(x, z) \le d(\varphi(x), \varphi(z)),
\]
for any $x, z \in X$ with $d(x,z)<\delta_X$. 

Let $U$ be a non-empty open subset of $X \cong \mathcal G_Y^{(0)}$. Since the set of eventually periodic points of $\varphi$ is dense, there exists $x_0 \in U$ which is eventually periodic. Therefore, there exists $k \in \N$ such that $\varphi^k(x_0)=x$ is periodic. Let $p$ denote the period of $x$. 

If necessary, after replacing $U$ with a smaller open set containing $x_0$, we  assume 
\begin{enumerate}
\item $\varphi^i(U) \cap Y = \emptyset$ for $i=0, 1, \ldots, k$ and
\item $(\varphi)|_{\varphi^i(U)} : \varphi^i(U) \rightarrow \varphi^{i+1}(U)$ is a homeomorphism for each $i=0, 1, \ldots, k-1$.
\end{enumerate}
Note that the condition that $Y$ is closed and does not contains any eventually periodic points is used to ensure that (1) holds. Condition (2) implies that 
\[
S=\{ (z, k, \varphi^k(z)) \mid z\in U \}
\]
is a $\mathcal{G}_{(X, \varphi)}$-set, and the first condition implies that it is also a $\mathcal{G}_Y$-set.

Continuing to follow the proof of \cite[Proposition 4.2]{A-D:PI}, let $W$ be a non-empty open subset of $\varphi^k(U)$ such that
\begin{enumerate}[resume]
\item $x\in W$,
\item $\varphi^p|_W$ is a homeomorphism onto its image, and
\item for $l=0, \ldots, p$ and $z_1, z_2 \in \varphi^l(W)$, we have that $d(z_1, z_2)< \delta_0$.
\end{enumerate}
As for the subset $U$, by replacing  $W$ with a smaller open set if necessary, (3), (4) and (5) still hold and in addition
\begin{enumerate}[resume]
\item $\varphi^i(W) \cap Y = \emptyset$ for $i=0, 1, \ldots, p$
\item $(\varphi)|_{\varphi^i(W)} : \varphi^i(W) \rightarrow \varphi^{i+1}(W)$ is a homeomorphism for each $i=0, 1, \ldots, p-1$.
\end{enumerate}
In particular, the set
\[ T=\{(\varphi^p(z), -p, z) \mid z \in W \} \]
is a $\mathcal{G}_Y$-set. 
Since $x\in \varphi^p(W) \cap \varphi^k(U)$ and these sets are open, there exists $r>0$ such that
\[
B(x, r) \subseteq \varphi^p(W) \cap \varphi^k(U).
\]
Since $x$ is periodic with period $p$, (5) implies that 
\[
\alpha_T(B(x, r))=\varphi^p_X(B(x,r)) \subseteq B(x, r).
\] 
This means that we can iterate $\alpha_T$ and have that, for each $l\in \N$, $B(x, r)$ is in the domain of $(\alpha_T)^l$. By abuse of notation, let $T^l$ denote the $\mathcal{G}_Y$-set associated to the map $(\alpha_T)^l$. 

By (5) together with the the fact that $x$ is not isolated, there exists $V_0 \subseteq B(x, r)$ an open set containing $x$ such that $(\alpha_T)^l(\overline{V}_0)\subseteq V_0$, but $(\alpha_T)^l(\overline{V}_0)\neq V_0$. Let $V= \alpha_S^{-1}(V_0)$. Then 
\begin{enumerate}[resume]
\item $V\subseteq U$ and
\item $(\alpha^l_X \circ \alpha_S)(\overline{V})\subseteq \alpha_S(V)$, but $(\alpha^l_X \circ \alpha_S)(\overline{V})\neq \alpha_S(V)$.
\end{enumerate}
Let \[ R:=\{ (w, k-lp, z) \mid w \in V \hbox{ and }z\in V_0 \hbox{ uniquely satisfies }\varphi^k(w)=\varphi^{lp}(z)\}.
\]
Then $R$ is a $\mathcal G_Y$-set. Moreover, (8) and (9) imply that $\overline{V} \subset s(R)$ and $\alpha_{S^{-1}}(\overline{V}) \subset V$ but $\alpha_{S^{-1}}(\overline{V}) \neq V$. Since $U$ was arbitrary, $\mathcal G_Y$ is locally contracting and the conclusion of the theorem now follows from \cite[Proposition 2.4]{A-D:PI}.
\end{proof}
\begin{remark}
    We note that this theorem can be generalized to the case when $(X, \varphi)$ is irreducible (rather than mixing). See Remark \ref{rem:irr-mix} for details.
\end{remark}

\begin{proposition} 
Let $X$ be a compact metric space with no isolated points, $\varphi: X\rightarrow X$ a mixing, open, locally expanding map and $w$ a fixed point. Let $Y \subset X$ be a compact subset
\begin{enumerate}
\item for every $x\in X$, $\orb(x)\cap  Y $  is the empty set or a singleton and
\item $Y_{\mathrm{ep}} = \{w\}$.
\end{enumerate}
Furthermore, suppose that the set of eventually periodic points of $\varphi$ is dense, but has empty interior. Then $C^*(\mathcal{G}_Y)$ is simple and purely infinite. 
\end{proposition}

\begin{proof}
The proof that $\mathcal G_Y$ is minimal, essentially principal, and locally contracting is essentially the same as the proof of Proposition~\ref{purelyInfProp}. It is therefore omitted. Then, again as in the proof of Proposition~\ref{purelyInfProp}, it follows from \cite[Proposition 2.4]{A-D:PI} that $C^*(\mathcal G_Y)$ is simple and purely infinite. 
\end{proof}

\begin{example}
    Let $E$ be a finite directed graph with primitive adjacency matrix $A$. (We also assume $A\neq[1]$). Let $\mathcal G$ be the Deaconu--Renault groupoid $\mathrm{C}^*$-algebra associated to $(\Sigma^+, \sigma_+)$. Using \cite[Example 2]{MR1233967}, we have that $C^*(\mathcal G)$ is isomorphic to the Cuntz--Krieger algebra $\mathcal O_A$, and so has $K$-theory
    \[
    K_0(C^*(\mathcal G)) \cong \coker(I - A), \quad K_1(C^*(\mathcal G)) \cong \ker(I - A). 
    \]
    In particular, $\rank(K_0(C^*(\mathcal G))) = \rank(K_1(C^*(\mathcal G)))$. Let $y \in \Sigma^+$ be a point which is not eventually periodic, and let $Y = \{ y\}$. Then $(\Sigma^+, \sigma_+)$ and $Y$ satisfy the conditions of Proposition~\ref{purelyInfProp}. Using Theorem~\ref{thm:BreExaSeq}, we have
    \[ K_0(C^*(\mathcal G_Y)) \cong \mathbb Z \oplus \coker(I - A), \quad K_1(C^*(\mathcal G_Y)) \cong  \ker(I - A). \]
    Since the ranks of $K_0$ and $K_1$ are not the same, the algebra $C^*(\mathcal G_Y)$ is not a Cuntz--Krieger algebra associated to a finite integer matrix. By choosing $Y = \{ y_1, \dots, y_d\}$ with each $y_i$ not eventually periodic and $\orb(y_i)$ pairwise disjoint, one obtains
      \[ K_0(C^*(\mathcal G_Y)) \cong \mathbb Z^d \oplus \coker(I - A), \quad K_1(C^*(\mathcal G_Y)) \cong  \ker(I - A). \]
\end{example}

This completes the theoretical development of orbit-breaking in Deaconu--Renault groupoids. In summary, we can orbit-break in a Deaconu--Renualt groupoid obtained from a surjective local homoeomorphism as long as the set $Y$ (where we are breaking orbits) contains no eventually periodic points and meets every (generalized) orbit at most once. If we furthermore assume that the surjective local homeomorphism satisfies some natural conditions, then the orbit-breaking subalgebra will be simple and purely infinite. Finally, there is a variant of the construction where we orbit-break in the groupoid obtained from the Deaconu--Renault groupoid by removing a fixed point (i.e., we can orbit-break in $\mathcal{G}|_U$ where $U=X \setminus \{w\}$, $w$ is a fixed point, and $\mathcal{G}$ is the Deaconu--Renault groupoid).

\section{Constructing $\mathcal{O}_2$} \label{sec:OBO2}
The rest of the present paper applies the orbit-breaking process to a particular class of systems. The goal is the construction of groupoid models for UCT Kirchberg algebras.

In particular, the goal of this section is the following: Given $n \in \N$, we construct a groupoid model for $\mathcal{O}_2$ where the groupoid is obtained from a locally expanding dynamical system where the relevant space has dimension $n$. More precisely, the space contains the $n$-torus as one of its connected components. This construction will allow us to orbit-break at quite general spaces by embedding into the $n$-torus. 

The starting point for the construction of our groupoid models for $\mathcal O_2$ is the following result, which is a special case of a theorem for the binary factors of shifts of finite type constructed by the second author \cite{MR4700629}.

\begin{theorem} \label{thm:caseIan}
There exists $X$ a finite dimensional, compact metric space and $\varphi: X\rightarrow X$ a continuous, surjective, local homeomorphism such that 
\begin{enumerate}
\item $X$ has no isolated points, $\varphi$ is mixing and locally expanding, the set of eventually periodic points of $\varphi$ is dense but has empty interior, 
\item the connected components of $X$ are circles and points and both cases occur, and
\item the groupoid $\mathrm{C}^*$-algebra associated to $(X, \varphi)$ is the Cuntz algebra $\mathcal{O}_2$, so in particular its $K$-theory is trivial in both degrees.
\end{enumerate}
\end{theorem}
\begin{proof}
Since $\mathcal O_2$ is a unital UCT Kirchberg algebra, by Theorem~\ref{thm:PutSys}, we need only find finite directed graphs $E$ and $F$ with adjacency matrices $A_E$ and $A_F$ such that 
\begin{enumerate}
    \item $A_E$ is primitive (that is, the shift of finite type associated to $A_E$ is mixing),
    \item $I-A_E^T$ and $I-A_F^T$ are invertible over $\Z$,
    \item there are embeddings $\xi^0, \xi^1 : F \to E$ satisfying (H0)--(H3).
\end{enumerate}
Then the the binary system $(X_\xi, \sigma_\xi)$ associated to $(\xi^1, \xi^2)$ will satisfy (1) and (2) and the associated Deaconu--Renault groupoid $\mathcal G$ will satisfy $K_*(C^*(\mathcal G)) \cong 0$. It follows from the Kirchberg--Phillips classification theorem that $C^*(\mathcal G) \cong \mathcal O_2$. Such graphs are easy to obtain, for example let $F$ and $E$ be the directed graphs given by adjacency matrices
\[
[2] \hbox{ and } \left[ \begin{array}{cc} 5 & 3 \\ 5 & 5 \end{array} \right],
\]
respectively. There are many other choices. For further details, we refer the reader to Example~\ref{Ex:O2Put}.  
\end{proof}

Based on the previous result, there exist locally expanding dynamical systems such that the groupoid $\mathrm{C}^*$-algebra associated to the Deaconu--Renault groupoid obtained from it is $\mathcal{O}_2$ and the unit space of the groupoid contains the circle. We will require a groupoid model whose unit space has arbitrarily large finite dimension. First, we require the next theorem about the $K$-theory the $\mathrm{C}^*$-algebras obtained from products of groupoids associated to binary systems.

\begin{theorem} \label{thm:productDR}
Suppose that $(X_1, \varphi_1)$ and $(X_2, \varphi_2)$ are binary systems associated to embeddings of finite graphs with adjacency matrices $A_{E_1}, A_{F_1}$ and $A_{E_2}$, $A_{F_2}$, respectively. Then $(X_1 \times X_2, \varphi_1 \times \varphi_2)$ has no isolated points, $\varphi_1\times \varphi_2$ is locally expanding, and the set of eventually periodic points is dense but has empty interior. Moreover, $K_0(C^*(\mathcal{G}))$ is isomorphic to  
\begin{align*} K_0(C^*(\mathcal{G})) \cong& \frac{\Z^{k_1}}{(I-A^T_{E_1} \otimes A^T_{E_2})\Z^{k_1}} \oplus \ker(I-A^T_{E_1}\otimes A^T_{F_2})\\
& \oplus \frac{\Z^{k_4}}{(I-A^T_{F_1} \otimes A^T_{F_2})\Z^{k_4}} \oplus \ker(I-A^T_{F_1}\otimes A^T_{E_2}) , 
 \end{align*}
and $K_1(C^*(\mathcal{G}))$ is isomorphic to 
\begin{align*} K_1(C^*(\mathcal{G})) \cong& \frac{\Z^{k_2}}{(I-A^T_{E_1} \otimes A^T_{F_2})\Z^{k_2}} \oplus {\mathrm ker}(I-A^T_{F_1}\otimes A^T_{F_2})\\
& \oplus  \frac{\Z^{k_3}}{(I-A^T_{F_1} \otimes A^T_{E_2})\Z^{k_3}} \oplus \ker(I-A^T_{E_1}\otimes A^T_{E_2}) ,
\end{align*}
where $\mathcal{G}$ is the Deaconu--Renault groupoid associated to $(X_1\times X_2, \varphi_1\times \varphi_2)$ and $k_i$ for $i=1, \ldots 4$ is determined by the sizes of the relevant matrices (when each matrix is $n$ by $n$, we get that $k_i=n^2$ for each $i$).
\end{theorem}
\begin{proof}
We have to check that $(X_1 \times X_2, \varphi_1 \times \varphi_2)$ has the following properties: 
\begin{enumerate}
\item $X_1\times X_2$ has no isolated points,
\item $\varphi_1\times \varphi_2$ is continuous, onto, a local homeomorphism, and mixing,
\item $\varphi_1\times \varphi_2$ is locally expanding, and
\item the set of eventually periodic points for $\varphi_1\times \varphi_2$ is dense but has empty interior.
\end{enumerate}
The proofs of the (1) and (2) are standard, so are omitted. Property (3) follows by taking $\gamma= {\mathrm{min} }\{ \gamma_1, \gamma_2 \}$ and $\lambda={\mathrm{min}}\{\lambda_1, \lambda_2\}$ where $\gamma_1$, $\gamma_2$, $\lambda_1$, $\lambda_2$ are the constants in Definition \ref{def:ExpMap}) for $(X_1, \varphi_1)$ and $(X_2, \varphi_2)$ respectively. Then, one checks that with this $\gamma$ and $\lambda$, the product system satisfies the definition of locally expanding (again, see Definition \ref{def:ExpMap}). Finally, Property (4) follows from the fact the $(x_1, x_2)$ is eventually periodic for $\varphi_1 \times \varphi_2$ if and only if $x_1$ is eventually periodic for $\varphi_1$ and $x_2$ is eventually periodic for $\varphi_2$.

We now consider the $K$-theory of $\mathcal G$. Let $\mathcal{G}^s_{X_1}$ and $\mathcal{G}^s_{X_2}$ be the stable groupoids of the Smale space associated to $(X_1, \varphi_1)$ and $(X_2, \varphi_2)$ respectively, as constructed in \cite[Section 4]{MR4700629}. Then basic Smale space theory implies that the stable groupoid associated to $(X_1 \times X_2, \varphi_1\times \varphi_2)$, denoted by $\mathcal{G}^s_{X_1\times X_2}$, is equal to $\mathcal{G}^s_{X_1}\times \mathcal{G}^s_{X_2}$ (where some care is required in terms of the choice of periodic points in the construction of the stable groupoids, see for example, \cite{PutSpi:Smale}). Hence
\[ C^*_r( \mathcal{G}^s_{X_1\times X_2}) =C^*_r(\mathcal{G}^s_{X_1}\times \mathcal{G}^s_{X_2}) \cong C^*_r(\mathcal{G}^s_{X_1}) \otimes C^*_r( \mathcal{G}^s_{X_2}),
\]
where the choice of tensor product is not relevant because the groupoids are amenable. Applying the K\"unneth formula and using the fact that the $K$-theories in this case are torsion free by \cite[Theorem 6.1]{MR4700629}, we get 
\begin{align*}
\lefteqn{K_0(C^*_r( \mathcal{G}^s_{X_1\times X_2}))} \\ &\cong \left( K_0(C^*_r(\mathcal{G}^s_{X_1})) \otimes K_0(C^*_r(\mathcal{G}^s_{X_2}))  \right) \oplus \left( K_1(C^*_r(\mathcal{G}^s_{X_1})) \otimes K_1(C^*_r(\mathcal{G}^s_{X_2}))  \right), \\
\lefteqn{K_1(C^*_r( \mathcal{G}^s_{X_1\times X_2}))} \\ & \cong \left( K_0(C^*_r(\mathcal{G}^s_{X_1})) \otimes K_1(C^*_r(\mathcal{G}^s_{X_2}))  \right) \oplus \left( K_1(C^*_r(\mathcal{G}^s_{X_1})) \otimes K_0(C^*_r(\mathcal{G}^s_{X_2}))  \right) .
\end{align*}
Since the isomorphisms respect the maps induces from $\varphi_1$, $\varphi_2$ and $\varphi_1 \times \varphi_2$, we can apply the Pimsner--Voiculescu six term exact sequence and \cite[Theorem 6.1]{MR4700629} along with its proof to get the required $K$-theory result. 
\end{proof} 

\begin{remark}
It is worth mentioning that the Deaconu--Renault groupoid associated to $(X_1 \times X_2, \varphi_{X_1}\times \varphi_{X_2})$ is not the product of the ones associated to $(X_1, \varphi_1)$ and $(X_2, \varphi_2)$.
\end{remark}

\begin{theorem} \label{thm:dimO2}
Given $n\in \N$, there exists $(X, \varphi)$ such that 
\begin{enumerate}
\item $X$ has no isolated points, $\varphi$ is  mixing, open, and locally expanding, the set eventually periodic points is dense, but has empty interior, 
\item the $n$-torus occurs as one of the connected component of $X$, and 
\item the groupoid $\mathrm{C}^*$-algebra associated to $(X, \varphi)$ is the Cuntz algebra $\mathcal{O}_2$, so in particular its $K$-theory is trivial in both degrees.
\end{enumerate}
\end{theorem}
\begin{proof}
For $n=1$, we can simply take $(X, \varphi)$ as in Theorem \ref{thm:caseIan}. For $n=2$, let $(X_1, \varphi_1)$ be the locally expanding system constructed as in Example~\ref{ex:IanBasEx} from the embeddings $\xi^0, \xi^1 : F_1 \to E_1$ where $F_1$ and $E_1$ are  defined by adjacency matrices
\[
A_{F_1}= [ 1 ] \hbox{ and } A_{E_1}=[3] .
\]
Next, let $(X_2, \varphi_2)$ be the binary system associated $\xi^0,  \xi^1 : F_2 \to E_2$, where $F_2$ and $E_2$ are given by adjacency matrices
\[
A_{F_2}=\left[ \begin{array}{cc} 2 & 1 \\ 1 & 1 \end{array} \right] \hbox{ and }A_{E_2}=\left[ \begin{array}{ccc} 5 & 3 & 1 \\ 3 & 3 & 1 \\ 12 & 10 & 4 \end{array} \right],
\]
where the existence of $\xi^0$ and $\xi^1$ follows from Remark \ref{rem:embH1H2}. Using Theorem \ref{thm:productDR}, we need to show that $\det(I-A_{F_1}\otimes A_{F_2})=1$, $\det(I-A_{E_1}\otimes A_{F_2})=1$, $\det(I-A_{F_1}\otimes A_{E_2})=1$, and $\det(I-A_{E_1}\otimes A_{E_2})=1$. This amounts to showing $\det(I-A_{F_2})=1$, $\det(I-3A_{F_2})$, $\det(I-A_{E_2})=1$, and $\det(I-3A_{E_2})=1$. This follows from direct calculation.

By Theorem \ref{thm:productDR} the $K$-theory of the groupoid $\mathrm{C}^*$-algebra associated to $(X_1\times X_2, \varphi_1\times \varphi_2)$ is trivial in both degrees. The Kirchberg--Phillips classification theorem then implies that it is $\mathcal{O}_2$. Finally, $X_1$ contains $S^1$ as one of its connected components and likewise for $X_2$, so $X_1 \times X_2$ contains the $2$-torus as one of its connected components. 

We generalize the previous construction to $n\ge 3$ as follows. As above, let $(X_1, \varphi_1)$ be the binary system associated with $\xi^0,  \xi^1 : F_1 \to E_1$, where $F_1$ and $E_1$ are given by adjacency matrices
\[
A_{F_1}= [ 1 ] \hbox{ and } A_{E_1}= [3],
\]
see Example \ref{ex:IanBasEx}.

 In addition, let $(X, \varphi)=(X_1 \times X_1 \times \ldots \times X_1 \times X_3, \varphi_1\times \varphi_1 \times \ldots \times \varphi_1 \times \varphi_3)$ where there are $(n-1)$-copies of $(X_1, \varphi_1)$ in the product that defines $(X, \varphi)$ and $(X_3, \varphi_3)$ is the binary system associated with adjacency matrices
\[
A_{F_3}=\left[ \begin{array}{ccc} 1 & 1 & 0 \\ 1 & 1 & 1 \\ 3^{n-1}-4 & 3^{n-1}-5 & 3^{n-1}-2 \end{array} \right] ,
\]
and  
\[ A_{E_3}=\left[ \begin{array}{ccc} 3 & 3 & 0 \\ 3 & 3 & 3 \\ 5\cdot 3^{n}-15 & 5 \cdot 3^{n}-14 & 3^{n+1}-6 \end{array} \right],
\]
where we note that Remark \ref{rem:embH1H2} implies that the required embeddings exist. 

In similar way to the proof of Theorem \ref{thm:productDR} and the case $n=2$ discussed above, we need to show that $\det(I-A_{F_3})=1$, $\det(I-A_{E_3})=1$, $\det(I-3^{n-1}\cdot A_{F_3})=1$, and $\det(I-3^{n-1}\cdot A_{E_3})=1$. This can be done by direct calculation (also see Remark \ref{rem:findMatrix} below). 

In summary, since $X=X_1 \times \ldots \times X_1\times X_3$, $X$ has the $n$-torus as one of its connected components. In addition, since the required matrices are invertible over the integers, the $\mathrm{C}^*$-algebra associated to the Deaconu--Renault of $(X, \varphi)$ has trivial $K$-theory in both degrees, so it is $\mathcal{O}_2$ by the Kirchberg--Phillips classification theorem.
\end{proof}
\begin{remark} \label{rem:findMatrix}
The matrix $A_{F_3}$ in the previous proof was found by considering the linear system obtained from the determinant conditions (for example, $\det(I-A_{F_3})=1$) for the matrix 
\[ \left[ \begin{array}{ccc} 1 & 1 & 0 \\ 1 & 1 & 1 \\ x & y & z \end{array} \right]. \] 
The matrix $A_{E_3}$ was found using the same method applied to the matrix
\[ \left[ \begin{array}{ccc} 3 & 3 & 0 \\ 3 & 3 & 3 \\ x & y & z \end{array} \right]. \] 
\end{remark}

\section{Orbit-breaking subalgebras of $\mathcal O_2$} \label{sec:OB}

In this section, we apply Theorem \ref{thm:BreExaSeq} to the system constructed in the previous section. In terms of the notation in Section \ref{sec:OBDR}, we use the orbit-breaking method where $Y_{ep}=Y \cap X_{ep}$ is the empty set. The resulting algebras are unital UCT Kirchberg algebras. 

\begin{theorem} \label{thm:mainInsideO2}
Suppose that $G_0$ and $G_1$ are countable abelian groups. Then there exists $(X, \varphi)$ and $Y \subseteq X$ such that 
\begin{enumerate}
\item the orbit-breaking groupoid $\mathcal{G}_Y$ is essentially principal and locally contracting, so that $C^*(\mathcal{G}_Y)$ is a unital UCT Kirchberg $\mathrm{C}^*$-algebra,
\item $K_0(C^*(\mathcal{G}_Y)) \cong \Z \oplus G_0$ and $K_1(C^*(\mathcal{G}_Y))\cong G_1$.
\end{enumerate}
\end{theorem} 
\begin{proof}
Given countable abelian groups $G_0$ and $G_1$, there exists finite dimensional $Y$ such that $K^0(Y) \cong \Z \oplus G_0$ and $K^1(Y) \cong G_1$. Take $(X, \varphi)$ as in Theorem \ref{thm:dimO2} with $n$ sufficiently large, so that there exists an embedding of $Y$ denoted by $\iota$ into $(n-1)$-dimensional torus. 

Using this embedding, we embed $Y$ into one of the connected components of $X$ that is homeomorphic to the $n$-torus as follows. We use the notation $\mathbb{T}^n=\{ (z_1, \ldots, z_n) \mid z_i \in \mathbb{T} \}$ for the relevant $n$-torus inside $X$. Recall that by construction $X$ is the product $X_1\times \ldots \times X_n$ and $\varphi= \varphi_1 \times \ldots \times \varphi_n$ (see the proof of Theorem \ref{thm:dimO2} for details). Take $z \in X_n$ that is not eventually periodic for $\varphi_n$ and embed $Y$ via $y \mapsto (\iota(y), z) \in \mathbb{T}^n \subseteq X$. 

We must show that 
\begin{enumerate} 
\item the image of $Y$ doesn't contain any eventually periodic points and  
\item each generalized orbit meets the image of $Y$ at most once.
\end{enumerate}
For the first item, suppose there exists $k$ and $l$ such that $\varphi^k(z_1, \ldots, z_{n-1}, z)=\varphi^l(z_1, \ldots, z_{n-1}, z)$. Then, in particular, $\varphi_n^k(z)=\varphi_n^l(z)$, but this is not possible because the little  wee $z$ we chose is not eventually periodic. For the second item, given $(\iota(y), z)$, we have that $\varphi^k( \iota(y), z)$ cannot be of the form $(\iota(\hat{y}), z)$ (for some $\hat{y} \in Y$ and $k\ge 1$) because that would imply that $\varphi_n^k(z)=z$, which again is not possible because $z$ is not eventually periodic. 

Proposition \ref{purelyInfProp} implies that $\mathcal{G}_Y$ is essentially principal and locally contracting, so that $C^*(\mathcal{G}_Y)$ is unital, Kirchberg and satisfies the UCT. Applying Theorem \ref{thm:BreExaSeq} to this situation gives 
\begin{displaymath} 
\xymatrix{ K^0(Y) \ar[r] & K_0(C^*(\mathcal{G}_Y)) \ar[r]^-{\iota_*} & \{0\} \ar[d]^{\partial_{\mathrm OB}}\\
\{ 0 \} \ar[u]^{\partial_{\mathrm OB}} & K_1(C^*(\mathcal{G}_Y)) \ar[l]_-{\iota_*} & K^1(Y), \ar[l]}
\end{displaymath}
where we have used the fact that $C^*(\mathcal{G}_{(X, \varphi)})$ is $\mathcal{O}_2$, so its $K$-theory is trivial in both degrees. It follows from this exact sequence that $K_0(C^*(\mathcal{G}_Y)) \cong K^0(Y) \cong \Z \oplus G_0$ and $K_1(C^*(\mathcal{G}_Y))\cong K^1(Y) \cong G_1$.
\end{proof}

\begin{theorem} \label{thm:Stabcont}
Suppose that $G_0$ and $G_1$ are countable abelian groups and $A$ is the unique stable UCT Kirchberg algebra with $K_0(A) \cong \Z \oplus G_0$ and $K_1(A) \cong G_1$. Then there exists $(X, \varphi)$ and $Y \subseteq X$ such that 
\begin{enumerate}
\item $C^*_r(\mathcal{G}_Y \times \mathcal{R})\cong A$ and
\item $C^*_r(\mathcal{G}_Y \times \mathcal{R}) \subseteq \mathcal{O}_{\infty}\otimes \mathcal{K} \subseteq \mathcal{O}_2 \otimes \mathcal{K}$
\end{enumerate}
where $\mathcal{R}$ is an amenable \'etale equivalence relation whose $\mathrm{C}^*$-algebra is the compact operators $\mathcal{K}$ (for example, take $\mathcal R$ to be $\N$ with the full equivalence relation).
\end{theorem}
\begin{proof}
By the Kirchberg--Phillips theorem, $K$-theory is a complete isomorphism invariant for stable UCT Kirchberg algebras. Let $(X, \varphi)$ and $Y\subseteq X$ be as in the previous theorem to obtain a unital UCT Kirchberg algebra $C^*_r(\mathcal{G}_Y)$ with the required $K$-theory groups. Then $C^*_r(\mathcal{G}_Y \times \mathcal{R})$ is a stable UCT Kirchberg algebra and hence is isomorphic to the given $A$. For any point non-eventually periodic point $x \in X$, we have $K_0(\mathcal G_{\{x\}}) = \mathbb Z$ and $K_1 (\mathcal G_{\{x\}})= 0$. In particular, this holds for $y$, and gives $C^*(\mathcal G_{\{y\}} \times \mathcal R) \cong \mathcal O_\infty \otimes \mathcal K$. Since since $Y$ contains $\{y\}$ as a closed subset, (2) now follows from Theorem \ref{thm:subsetOrbitBreak}.
\end{proof}

\begin{theorem} \label{thm:maps}
   Suppose $G_0, H_0$ and $G_1, H_1$ are countable abelian groups and $\psi_* : G_* \to H_*$ is a group homomorphism. Then there exist compact metric spaces $Y_1$ and $Y_2$, and a continuous embedding $\alpha : Y_1 \to Y_2$ such that 
\begin{displaymath} 
\xymatrix{ \tilde{K}^*(Y_2) \ar[d]_{\cong} \ar[r]^{\alpha_*} & \tilde{K}^*(Y_1) \ar[d]^{\cong} \\
G_* \ar[r]_{\psi} & H_* .
}
\end{displaymath}
Note that $\tilde{K}$ denotes the reduced $K$-theory, so $K^0(Y_2) \cong \mathbb Z \oplus G_0$, $K^1(Y_2) \cong G_1$ and likewise for $Y_1$.
\end{theorem}

\begin{proof}
  First, observe that it is enough to construct a continuous map $\alpha$ which is not necessarily an embedding. Indeed, if $\alpha : Y_1 \to Y_2$ is not already an embedding, let $d \in \mathbb N$ be large enough so that there exists an embedding $\iota : Y_2 \into [0,1]^d$. Given $\alpha : Y_1 \to Y_2$, let 
    \[ \tilde{\alpha} : Y_1 \to Y_2 \times [0,1]^d, \quad y \mapsto (\alpha(y), \iota(y)).\]
    Then $\tilde{\alpha}$ is an embedding and $\tilde{\alpha}^*$ agrees with $\alpha^*$ via the projection map $p : Y_2 \times [0,1]^d \to Y_2$ .

  Next, observe that by taking a wedge products of the space $Y^{(0)}_i$ with the suspension of the space $Y^{(1)}_i$, $i = 1, 2$, where  $\tilde{K}_0(Y^{(0)}_1) = H_0$,  $\tilde{K}^0(Y^{(1)}_1) = H_1$, $\tilde{K}_0(Y^{(0)}_2) = G_0$, $\tilde{K}_0(Y^{(1)}_2) = G_1$ and $K_1(Y^{(j)}_i) = 0$, $j = 0,1$, $i=1, 2$, it is enough to show the result holds when $G_1 = H_1 = \{0\}$. 
  
Suppose that $G_0$ and $H_0$ are finitely generated. Let 
    \[ \xymatrix{0 \ar[r] & F_0 \ar[r] & F_1 \ar[r] & G_0  \ar[r]  & 0}
    \]
    and
     \[ \xymatrix{0 \ar[r] & E_0 \ar[r] & E_1 \ar[r] & H_0  \ar[r]  & 0}
    \]  
    be free resolutions of $G_0$ and $H_0$, respectively. We note that $F_0, F_1, E_0, E_1$ are each finitely generated (and, of course, free). 
    
    Since $\psi : G_0 \to H_0$ is a group homomorphism, and $E_0, E_1$ are free and hence projective, we obtain a commutative diagram
   \[ \xymatrix{0 \ar[r]  & F_0 \ar[r]^{\gamma} \ar[d]_{\phi_0} & F_1 \ar[d]_{\phi_1} \ar[r] & G_0  \ar[r] \ar[d]^{\psi}  & 0 \\
  0 \ar[r] & E_0 \ar[r]_\kappa & E_1 \ar[r] & H_0  \ar[r]  & 0  .}
    \]  

     Define $X_i$ to be a wedge of $\rank(F_i)$-many copies of $S^2$, and similarly, let $Z_i$ be a wedge of $\rank(E_i)$ copies of $S^2$. Then $\tilde{K}^0(X_i) = F_i$ and $\tilde{K}^0(Z_i) = E_i$. Let $g : X_1 \to X_0$, $h : Z_1 \to Z_0$ and $f_i : Z_i \to X_i$, $i = 1, 2$ be continuous maps inducing the $\gamma, \kappa, \phi_0, \phi_1$ in the commutative diagram.

     Let $Y_1$ be the mapping cone of $h$ and $Y_2$ the mapping cone of $g$. Then we obtain a continuous map $\alpha : Y_1 \to Y_2$ with $K_0(\alpha) = \psi$ such that the diagram
     
    \[ \xymatrix{X_0  &\ar[l]_{g}  X_1   & \ar[l] Y_2 \\
  Z_0 \ar[u]^{f_0} &  \ar[l]^h Z_1  \ar[u]^{f_1} & Y_1  \ar[u]_\alpha \ar[l]}
    \]   
commutes.

If $G_0, H_0$ are countably generated, find $\varinjlim G^{(n)} = G_0$ and $\varinjlim H^{(n)} = H_0$ where each $G^{(n)}$, respectively $H^{(n)}$ is finitely generated. 

For $n \in \mathbb N$, apply the above argument  to obtain the diagram 
   \[ \xymatrix{X^{(n)}_0  &\ar[l]_{g^{(n)}}  X^{(n)}_1   & \ar[l] Y_2^{(n)} \\
  Z^{(n)}_0 \ar[u]^{f^{(n)}_0} &  \ar[l]^{h^{(n)}} Z^{(n)}_1  \ar[u]^{f^{(n)}_1} & Y^{(n)}_1  \ar[u]_{\alpha^{(n)}} \ar[l]}.
    \]   
At stage $n+1$, choose the free resolutions so that  $X^{(n+1)}_i$ and $Z^{(n+1)}_i$ are obtained by attaching finitely many cells. Then we have an induced map $Y_i^{(n)} \to Y_i^{(n+1)}$ which is obtained by attaching finitely many cells, and where $\alpha^{(n+1)}|_{Y_1^{(n)}}= \alpha^{(n)}$.
Let $Y_i = \varinjlim (Y_i^{(n)} \to  Y_i^{(n+1)})$. Each $Y_i$ is a CW-inductive limit, via CW-inclusions, with no cells of dimension greater than $3$ being attached. Thus $Y_i$ is a CW complex with covering dimension at most $3$.  By continuity of reduced $K$-theory for CW-direct limits,
\[
K_0(Y_1) = \varinjlim K_0(Y_1^{(n)}) = H_0, \qquad K_0(Y_2) = \varinjlim K_0(Y_2^{(n)}) = G_0.
\]
Since the maps $\alpha^{(n)}$ are compatible with the structure maps, they induce a continuous map $\alpha = \varinjlim \alpha^{(n)} : Y_1 \to Y_2$, which is the desired map.
\end{proof}

Note that in the next theorem, the embeddings obtained do not require the UCT. 

\begin{theorem} \label{thm:inclusions}
    Suppose $G_0, H_0$ and $G_1, H_1$ are countable abelian groups and $\psi_* : G_* \to H_*$ is a group homomorphism. Then there exists $(X, \varphi)$ and $Y_1, Y_2$ finite-dimensional compact metric spaces such that 
    \begin{enumerate}
        \item  $X$ has no isolated points, $\varphi$ is irreducible, open and locally expanding, and $X_{\mathrm{ep}}$ is dense and has empty interior,
  \item there are embeddings $Y_1 \into Y_2 \into X$ where the image of $Y_2$ meets every orbit at most once and contains no eventually periodic points, 
  \item $C^*(\mathcal G_{Y_2})$ is a subalgebra of $C^*(\mathcal G_{Y_1})$,
  \item the diagrams
\begin{displaymath} 
\xymatrix{ K_0(C^*(\mathcal G_{Y_2})) \ar[d]_{\cong} \ar[r]^{\iota_*} & K_0(C^*(\mathcal G_{Y_1})) \ar[d]^{\cong} \\
\mathbb{Z} \oplus G_0 \ar[r]_{\id \oplus \psi_*} & \mathbb Z \oplus H_0,} \quad \xymatrix{ K_1(C^*(\mathcal G_{Y_2})) \ar[d]_{\cong} \ar[r]^{\iota_*} & K_1(C^*(\mathcal G_{Y_1})) \ar[d]^{\cong} \\
 G_1 \ar[r]_{\psi_*} & H_1 ,
}
\end{displaymath}
commute, where $\iota : C^*(\mathcal G_{Y_2}) \into C^*(\mathcal G_{Y_1})$ is the inclusion map.
   \end{enumerate} 
\end{theorem}

\begin{proof}
    By Theorem~\ref{thm:maps}, there are compact, finite-dimensional metric spaces $Y_1, Y_2$ and a continuous embedding $\alpha : Y_1 \to Y_2$ such that 
    \begin{displaymath} 
\xymatrix{ \tilde{K}^*(Y_2) \ar[d]_{\cong} \ar[r]^{\alpha^*} & \tilde{K}^*(Y_1) \ar[d]^{\cong} \\
G_* \ar[r]_{\phi} & H_* .
}
\end{displaymath}
Let $(X, \varphi)$ be the system realizing $\mathcal O_2$ as in Theorem~\ref{thm:dimO2} with $n$ large enough so that $Y_2$ can be embedded into an $(n-1)$-dimensional torus. Then $(X, \varphi)$ satisifies (1). By Theorem~\ref{thm:maps}, there is an embedding $\alpha : Y_1 \to Y_2$, and from this the proof of Theorem~\ref{thm:mainInsideO2} provides an embedding of $Y_1$, and hence an embedding of $Y_2$, into $X$ that satisfy (2). Viewing $Y_1 \subset Y_2 \subset X$, Theorem~\ref{thm:subsetOrbitBreak} provides an inclusion $\iota : C^*(\mathcal G_{Y_2}) \into  C^*(\mathcal G_{Y_1})$, so that (3) holds. Finally, (4) follows by combining Theorem~\ref{thm:maps}, Theorem~\ref{thm:mainInsideO2} (2) and the diagram
\begin{displaymath} 
\xymatrix{ \ar[r] & K^0(Y_2) \ar[d]_{\alpha^*} \ar[r] & K_0(C^*(\mathcal G_{Y_2})) \ar[d]^{\iota_*} \ar[r] & K_0(C^*(\mathcal{G})) \ar[d]_{=} \ar[r] & K^1(Y_2) \ar[d]_{\alpha^*} \ar[r] & \\
\ar[r] & K^0(Y_1) \ar[r] & K_0(C^*(\mathcal G_{Y_1})) \ar[r] & K_0(C^*(\mathcal{G})) \ar[r] & K^1(Y_1) \ar[r] &
,}
\end{displaymath}
which is commutative.
\end{proof}

 \begin{remark} Let $Y_1 \to Y_2$ be an embedding and let $A = C^*(\mathcal G_{Y_2})$ and $B = C^*(\mathcal G_{Y_1})$ be as in  Theorem~\ref{thm:inclusions}. Let $y  \in Y_1$. Then by the properties of $Y_1$, $y$ is not eventually periodic, so we can form the unital UCT Kirchberg algebra $C^*(\mathcal G_{\{y\}})$. Since $\{ y\} \into Y_1 \into Y_2$, we have open inclusions of groupoids $\mathcal G_{Y_2} \subset \mathcal G_{Y_1} \subset \mathcal G_{\{y\}} \subset \mathcal G$ giving inclusions of $\mathrm{C}^*$-algebras $C^*(\mathcal G_{Y_2}) \subset C^*(\mathcal G_{Y_1}) \subset C^*(\mathcal G_{\{y\}}) \subset C^*(\mathcal G) \cong \mathcal O_{2}$.
\end{remark}

As above, let $\mathcal{R}$ be an amenable \'etale groupoid with $\mathrm{C}^*$-algebra isomorphic to the compact operators (for example, let $\mathcal R$ be $\N$ with the full equivalence relation). Applying the results of this section and taking the product of $\mathcal G$ and $\mathcal G_Y$ with $\mathcal R$, we arrive at the next example.

\begin{example} \label{ex:subalgebraStableCase}
Let $Y_1 \to Y_2$ be an embedding and let $A = C^*(\mathcal G_{Y_2})$ and $B = C^*(\mathcal G_{Y_1})$ be as in  Theorem~\ref{thm:inclusions} and let $y  \in Y_1$. Then $K_0(C^*(\mathcal G_{\{y\}})) \cong \Z$ and $K_1(C^*(\mathcal{G})) \cong \{ 0 \}$ hence $C^*(\mathcal G_{\{y\}}\times \mathcal{R}) \cong \mathcal{O}_{\infty} \otimes \mathcal{K}$ by the Kirchberg--Phillips classification theorem. Thus
\[ C^*(\mathcal G_{Y_2}) \otimes \mathcal K \subset  C^*(\mathcal G_{Y_1}) \otimes \mathcal K \subset \mathcal O_\infty \otimes \mathcal K \subset \mathcal O_2 \otimes \mathcal K.\]
\end{example}

\begin{example} \label{ex:subgroup}
As a special case of Theorem \ref{thm:inclusions}, we have the following. Suppose that $G_0$ and $G_1$ are countable abelian groups and $A$ is the unique stable UCT Kirchberg algebra with $K_0(A) \cong  \Z \oplus G_0$ and $K_1(A) \cong G_1$. Also suppose that $H_0$ and $H_1$ are subgroups of $G_0$ and $G_1$ respectively and $B$ is the unique stable UCT Kirchberg algebra with $K_0(B) \cong  \Z \oplus H_0$ and $K_1(B) \cong H_1$. Then $B$ is isomorphic to a subalgebra of $A$.
\end{example}

\begin{example} \label{ex:Oinfinity}
    Let $Y$ be a finite-dimensional compact metric space and fix $y\in Y$. Consider
    \[ \{ y \} \subseteq Y \subseteq [0,1]^d \]
    where $d$ is taken to be large. Applying the process in Example \ref{ex:subalgebraStableCase}, we get
   \[ \mathcal O_\infty \otimes \mathcal K \subset  C^*(\mathcal G_{Y}) \otimes \mathcal K \subset \mathcal O_\infty \otimes \mathcal K \subset \mathcal O_2 \otimes \mathcal K.
   \]
\end{example}

Next we consider another way to construct groupoid models in the purely infinite case from those in the stably finite case. Let $A$ be a simple, separable, unital, nuclear $\mathrm C^*$-algebra. Then $A \otimes \mathcal O_\infty$ is a UCT Kirchberg algebra. Thus by taking a product with a groupoid whose $\mathrm{C}^*$-algebra is stably finite, we can produce further UCT Kirchberg algebras.

\begin{theorem}
Suppose that $\mathcal{G}_{\mathrm{sf}}$ is an amenable \'etale groupoid such that
\begin{enumerate}
\item $C^*(\mathcal{G}_{\mathrm{sf}})$ is simple, separable, stably finite, and
\item $K_0(C^*(\mathcal{G}_{\mathrm{sf}}))\cong G_0$ and $K_1(C^*(\mathcal{G}_{\mathrm{sf}}))\cong G_1$. 
\end{enumerate}
Let $Y$ be such that $\mathcal{G}_Y \times \mathcal{R}$ is a model for $\mathcal{O}_{\infty} \otimes \mathcal{K}$ (for let $Y$ be a single point which is not eventually periodic, see Example \ref{ex:Oinfinity}). Then $C^*(\mathcal{G}_{\mathrm{sf}} \times \mathcal{G}_Y\times \mathcal{R})$ is isomorphic to the unique stable UCT Kirchberg algebra $A$ with $K_0(A) \cong G_0$ and $K_1(A) \cong G_1$.
\end{theorem}

Using orbit-breaking groupoids for a minimal dynamical system $(X, \varphi)$ (here $X$ is a compact metric space and $\varphi: X \rightarrow X$ is a minimal homeomorphism)  \cite[Corollary 7.4]{DPSmain}, we obtain groupoids $\mathcal{G}_{\mathrm{sf}}$  These models give stable UCT Kirchberg algebras where the only restriction is that the $K_0$-group is of the form $G_0 \oplus T$ where $G_0$ is a simple dimension group and $T$ is a countable abelian group.

\section{orbit-breaking after removing a fixed point} \label{sec:OBO2LC} 

In this section, we generalize the construction beyond the compact setting. In particular, here we start with a dynamical system $(X, \varphi)$ as before, but remove a fixed point $w$ from $X$ so that the resulting $\mathrm{C}^*$-algebras are non-unital. In the notation of Section \ref{sec:OB}, we use the orbit-breaking method where $Y_{ep}=Y\cap X_{ep} = \{ w \}$.

Let $Y$ be a locally compact subspace (with one-point compactification $Y\cup\{ \infty\}$) such that
\begin{enumerate}
\item there exists a continuous injection $\iota : Y \cup \{\infty\} \to X$ such that $\iota(\infty) = w$,
\item $\iota(Y\cup \{\infty\})$ intersects each (generalized) orbit at most once (that is, for each $x\in X$, $\orb(x)\cap \iota(Y \cup \{\infty\})$ is the empty set or a singleton) and
\item $\iota(Y)$ contains no eventually periodic points.
\end{enumerate}
We identify $Y$ with a closed subset of $X \setminus \{w\}$. Recall that in Section \ref{sec:OBDR} we define $\mathcal{G}_Y$ using the following conditions: Given $(x, p, z) \in \mathcal{G}|_U$, then we let $(x, p, z) \in \mathcal{G}^Y_{(X, \varphi)}$ if
\begin{enumerate}
\item[(I)] $\orb(x)\cap Y$ is the empty set or
\item[(II)] $\orb(x)\cap Y=\{y\}$ for some $y\in Y$ and $x, z \in \cup_{l\in \N}\varphi^{-l}(y)$ or
\item[(III)] $\orb(x)\cap Y=\{y\}$ for some $y\in Y$ and $x, z \in \orb(x) \backslash \cup_{l\in \N}\varphi^{-l}(y)$.
\end{enumerate}

\begin{lemma} \label{lem:embed-LC-Y}
Let $Y$ be a finite-dimensional second countable locally compact metric space. There exists $n \in \mathbb N$ such that the following holds: Let $(X, \varphi)$  be the system constructed in Theorem~\ref{thm:dimO2} with an $n$-torus as a connected component. Then there is a fixed point $w \in X$ and an embedding $\iota : Y \cup \{ \infty \} \to X$ satisfying
\begin{enumerate}
\item $\iota (\infty) = w$,
\item $\iota(Y)$ contains no eventually periodic points, and meets every $\varphi$-orbit at most once.
\end{enumerate}
\end{lemma}

\begin{proof}
 The embedding of $Y$ is similar to that of the compact case, but since the point at infinity is mapped to a fixed point, we need to take extra care to ensure the embedding does not contain any eventually periodic points. Note that the one-point compactification $Y \cup \{\infty\}$ of $Y$ is metrizable. Thus there is a system $(X, \varphi)$ as constructed in Theorem~\ref{thm:dimO2} with $k$ sufficiently large so that there exists an embedding
 \[ j : Y \cup \{\infty\} \into (S^1)^k.
 \]
Then $X$ contains $(S^1)^k$ as one of its connected components and, by construction, $\varphi|_{S^1}$ is the two-fold covering map $z \to 2z \mod 1$. Let  $p \in X_3$ be any fixed point, and let
 \[ w := (\underbrace{0, \dots, 0}_k, \underbrace{0, \dots, 0}_k, w) \in (X_1)^{2k} \times X_3 = X.
 \]
 Then $w \in X$ is a fixed point. Since $k$ is sufficiently large, there is an embedding 
 \[ j : Y \cup \{\infty\} \to (S^1)^k,\]
 such that $j(\infty) = (0, \dots, 0)$. Let $\alpha$ be an irrational number. Set
 \[ \iota : Y \cup \{\infty\} \to X,  y \to (j(y), \alpha j(y),  j(y)^2, \alpha j(y)^2, p),
 \]
 Then $\iota$ is evidently an embedding and $\iota(\infty) = (j(\infty), j(\infty)0, p) =  w$. 

By Lemma~\ref{lem:GoodY}, we only need to show that if there are $k, l \in \mathbb Z$ and $z, z' \in \iota(Y)$ such that $\varphi^k(z) = \varphi^l(z')$, then $k-l = 0$.

Let $z, z' \in \iota(Y)$ and suppose that $\varphi^k(z) = \varphi^l(z')$ for some $k \neq  l \in \mathbb Z_{\geq 0}$. 
Write 
\[  z = (z_1, \dots, z_k, \alpha z_1, \dots, \alpha z_k,   z_1^2, \dots,  z_k^2,  \alpha z_1^2, \dots, \alpha z_k^2, p),\]
and
\[ z' = (z_1', \dots, z_k', \alpha z_1', \dots, \alpha z_k',   (z_1')^2, \dots,  (z_k')^2 , \alpha (z_1')^2, \dots, \alpha (z_k')^2,   p),
\]
for $ 0 \leq z_i, z_i' < 1$.  Then $2^{k-l} z_i - z_i' \in \mathbb Z$ and $2^{k-l} \alpha (z_i - z_i') \in \mathbb Z$. Since $\alpha$ is irrational, $2^{k-l} z_i = z_i'$.  Also, $(2^{k-l} - 1) z_i^2 \in \mathbb Z$ and $(2^{k-l} - 1) \alpha z_i^2 \in \mathbb Z$  so there exists an integer $m, n$ such that
\[ z_i' = \frac{m}{\sqrt{2^{k-l} - 1}} = \frac{n}{\sqrt{\alpha}}.\]
This is only possible if $z_i'  = 0$ for every $1 \leq i \leq k$, in which case $z' = w \notin \iota(Y)$, contradiction. Thus $k = l$, $z =z'$ and $\iota(Y)$ meets every orbit at most once.
\end{proof}

\begin{theorem} \label{thm:stableModelFixed}
Let $G_0$ and $G_!$ be countable abelian groups. Then there exists $(X, \varphi)$ and a locally compact, non-compact metric space $Y$ such that 
\begin{enumerate}
\item the orbit-breaking groupoid $\mathcal G_Y$ is essentially principal and locally contracting, 
\item $K_0(C^*(\mathcal G_Y)) \cong G_0$ and $K_1(C^*(G_Y)) \cong G_1$.
\end{enumerate}
In particular, $C^*(\mathcal G_Y)$ is isomorphic to the stable UCT Kirchberg algebra $A$ with $K_*(A) \cong G_*$.
\end{theorem}

\begin{proof}
The proof is a modification of Theorem~\ref{thm:mainInsideO2}. Let $Y$ be a locally compact, non-compact finite-dimensional metric space with $K^*(Y) \cong G^*$ (for example, let $Y_0$  be a compact finite-dimensional metric space satisfying $K^0(Y_0) \cong \mathbb Z \oplus G_0$ and $K^1(Y_0) = G_1$, and let $Y = Y_0 \setminus \{y\}$ for some point $y$). By Lemma~\ref{lem:embed-LC-Y}, There exists $n \in \mathbb N$ such that the following holds: Let $(X, \varphi)$  be the system constructed in Theorem~\ref{thm:dimO2} with an $n$-torus as a connected component. Then there is a fixed point $w \in X$ and an embedding $\iota : Y \cup \{ \infty \} \to X$ satisfying
\begin{enumerate}
\item $\iota (\infty) = w$,
\item $\iota(Y)$ contains no eventually periodic points, and meets every $\varphi$-orbit at most once.
\end{enumerate}
Let $U = X \setminus \{ w\}$, and let $\mathcal G$ be the Deaconu--Renault groupoid associated to $(X, \varphi)$. By Proposition~\ref{prop:GUopen}, $C^*(\mathcal G|_U)$ is a hereditary $\mathrm{C}^*$-subalgebra of $C^*(\mathcal G)$, which is moreover full since $C^*(\mathcal G)$ is simple. Thus $K_*(C^*(\mathcal G|_U)) \cong K_*(C^*(\mathcal G))$, which is trivial in both degrees. The result follows from Theorem~\ref{thm:BreExaSeq}; compare with the proof of Theorem~\ref{thm:mainInsideO2}.
\end{proof}

In the unital case of the previous section, we were unable to construct all unital Kirchberg algebras, because our method does not allow us to control the class of the unit in $K$-theory and the $K_0$-group needed to contain an copy of the integers. 

On the other hand, the method of removing a fixed point \emph{does} allow us to realize all stable UCT Kirchberg algebras, as the next theorem shows.

\begin{theorem} \label{thm:NonUnitalIncl}
       Let $A$ and $B$ be stable UCT Kirchberg algebras with 
       \[K_0(A) = G_0, K_1(A) = G_1,  \text{ and } \, K_0(B) = H_0, K_1(B)  = H_1.
       \]
      Suppose $\psi_* : G_* \to H_*$ is a group homomorphism.  
       \begin{enumerate}
           \item There exists $(X, \varphi)$ and $Y_1, Y_2$ finite-dimensional locally compact metric spaces such that $A \cong C^*(\mathcal G_{Y_2})$ and $B \cong C^*(\mathcal G_{Y_1})$.
           \item There is an open inclusion of groupoids $\mathcal G_{Y_2} \subset \mathcal G_{Y_1}$.
           \item There is an inclusion of $\mathrm{C}^*$-algebras $\Psi: A \to B$ such that $K_*(\Psi) = \psi_*$. In particular, $A$ is isomorphic to a $\mathrm{C}^*$-subalgebra of $B$. 
       \end{enumerate}
\end{theorem}
\begin{proof}
  Choose finite-dimensional compact metric spaces $\hat{Y}_1$ and $\hat{Y}_2$ as in Theorem~\ref{thm:inclusions} with respect to $\psi_* : G_* \to H_*$.  Let $y_1 \in \hat{Y}_1$, $y_2 \in \hat{Y}_1$. Set $Y_1 := \hat{Y}_1 \setminus \{y_1\}$ and $Y_2 := \hat{Y}_2 \setminus \{y_2\}$. By Theorem~\ref{thm:inclusions}, there is an embedding $\hat{Y}_1 \to \hat{Y}_2$. Restricting to an embedding of $Y_1 \to Y_2$, we obtain $\mathcal G_{Y_2}$ as an open subgroupoid of $\mathcal G_{Y_1}$ and thus an inclusion $\iota : C^*(\mathcal G_{Y_2}) \to C^*(\mathcal G_{Y_1})$. Using an analogous argument in the nonunital setting to that given in the proof ofTheorem~\ref{thm:inclusions} (4),  the diagrams 
  \begin{displaymath} 
\xymatrix{ K_0(C^*(\mathcal G_{Y_2})) \ar[d]_{\cong} \ar[r]^{\iota_*} & K_0(C^*(\mathcal G_{Y_1})) \ar[d]^{\cong} \\
 G_0 \ar[r]_{\id \oplus \psi_*} &  H_0,} \quad \xymatrix{ K_1(C^*(\mathcal G_{Y_2})) \ar[d]_{\cong} \ar[r]^{\iota_*} & K_1(C^*(\mathcal G_{Y_1})) \ar[d]^{\cong} \\
 G_1 \ar[r]_{\psi_*} & H_1 ,
}
\end{displaymath}
commute. This shows that $A \cong C^*(\mathcal G_{Y_2})$, $B \cong  C^*(\mathcal G_{Y_1})$, and $K_*(\Psi) = \psi_*$.
\end{proof}

\begin{remark}
    Using methods similar to those used in the previous section, we obtain results similar to Examples \ref{ex:subalgebraStableCase}, \ref{ex:subgroup}, and \ref{ex:Oinfinity}.
\end{remark}

\begin{example}
    We assume the setup of Theorem \ref{thm:stableModelFixed}, but make the additional assumption that $G_0\cong \Z \oplus \tilde{G}_0$ for some countable abelian group $\tilde{G}_0$. 

    Then, in the proof of Theorem \ref{thm:stableModelFixed}, let $Y$ be a \emph{compact} metric space with $K$-theory 
    \[  K^0(Y) \cong G_0\cong \Z \oplus \tilde{G}_0 \hbox{ and }K^1(Y)\cong G_1\]

    The one point compactification of $Y$ is $Y \cup \{ \infty\}$ where $\infty$ is an isolated point. The existence of the required embedding of $Y\cup \{\infty \}$ simplifies in this case because $\infty$ is isolated (compare the proof of Lemma \ref{lem:embed-LC-Y} with the proof of the embedding in the proof of Theorem \ref{thm:mainInsideO2}).   
\end{example}

\begin{example}
    Let $\mathcal{O}_n$ denote the Cuntz algebra for some $2<n<\infty$ and $\mathcal{K}$ denote the compact operators. Theorem \ref{thm:Stabcont} fails to give a groupoid model for $\mathcal{O}_n\otimes \mathcal{K}$ because the $K_0$-group of this algebra does not contain the integers.

    Theorem \ref{thm:stableModelFixed} does produce groupoid models for this algebra. An explicit choice for the relevant space is obtained as follows. Let $Y_0$ be the Moore space with $K$-theory:
    \[ K^0(Y_0) \cong \Z \oplus \Z_{n-1} \hbox{ and } K^1(Y_0)\cong \{ 0 \}. \]
    Then the required space is obtained by taking $Y$ to be $Y_0$ with one point removed (as was done more generally in the proof of Theorem \ref{thm:stableModelFixed}). 
\end{example}

\begin{example}
    Since there are many finite-dimensional (locally) compact metric spaces with the same $K$-theory groups, we obtain many different groupoid models for a given stable UCT Kirchberg algebra. 

    By varying the choice of $(X, \varphi)$ in the statement of Theorem \ref{thm:stableModelFixed}, we also get groupoid models where the unit spaces have different dimensions. 

    As an explicit example, let $A=\mathcal{O}_\infty \otimes \mathcal{K}$. For $Y$ (see Theorem \ref{thm:stableModelFixed}) we can take a single point or $\R^2$ or $\R^4$ (among (many) other choices). Moreover, once we have picked the space $Y$, there are many choices for $(X, \varphi)$ since we need only pick $(X, \varphi)$ so that it contains a $n$-torus where $n$ is large enough to admit the required embedding of $Y\cup \{\infty\}$, see the proofs of Lemma \ref{lem:embed-LC-Y} and Theorem \ref{thm:stableModelFixed}.
\end{example}


\end{document}